\documentclass[10pt, a4paper]{amsart}
\usepackage{amsthm,amssymb, amscd}
\usepackage{mathrsfs, verbatim}
\theoremstyle{plain}
\newtheorem{thm}{Theorem}
\newtheorem{prop}{Proposition}
\newtheorem{cor}[prop]{Corollary}
\newtheorem{lm}[prop]{Lemma}

\theoremstyle{definition}

\newtheorem{definition}{Definition}

\newtheorem*{conclusion}{Conclusion}

\theoremstyle{remark}
\newtheorem{rmk}{Remark}

\newtheorem{example}{Example}
\newtheorem*{notation}{Notation}
\newtheorem*{acknowledgment}{Acknowledgment}

\def\OO{\mathcal{O}}

\def\HH{\mathcal{H}}
\def\II{\mathcal{I}}

\def\N{\mathcal{N}}

\def\C{\mathbb{C}}

\def\PP{\mathbb{P}}
\def\Z{\mathbb{Z}}

\def\T{\mathcal{T}}

\begin{document}
\title[Normal bundles of rational curves]{Irreducible Components of  Hilbert Schemes of Rational Curves  with given Normal Bundle}
\author{A. Alzati} 
\author{R. Re}
\address{Alberto Alzati,  Dipartimento di Matematica F.Enriques, Universit\`a di Milano,
via Saldini 50, 20133 Milano (Italy)}
\email{alberto.alzati@unimi.it}
\address{Riccardo Re, Dipartimento di Matematica e Informatica, Universit\`a di Catania, viale Andrea Doria 6, 95125 Catania (Italy)}
\email{riccardo@dmi.unict.it}
\noindent
\dedicatory{To Rosario Strano, on his 70's birthday}
\date{9 February, 2015}
\maketitle
\begin{abstract} We develop a new general method for computing the decomposition type of the normal bundle to a projective rational curve. This method is then used to detect and explain an example of a reducible Hilbert scheme parametrizing all the rational curves in $\PP^s$ with a given decomposition type of the normal bundle. We also characterize smooth non-degenerate rational curves contained in rational normal scrolls in terms of the splitting type of their restricted tangent bundles and compute their normal bundles. \end{abstract}
\section{Introduction}
The projective rational curves $C\subset \PP^s$ of degree $d$ form a  quasi-projective irreducible subscheme $\HH^{\rm{rat}}_{d,s}$ of the Hilbert scheme of $\PP^s$. Any of these curves is the image of a birational map $f:\PP^1\to\PP^s$, defined up a automorphism of $\PP^1$. If one restricts oneself to rational curves  with ordinary singularities,  one may classify these curves by considering the splitting types as a direct sum of line bundles of the vector bundles $f^\ast \T_{\PP^s}$ and $\N_f=f^\ast \T_{\PP^s}/\T_{\PP^1}$, commonly called the {\em restricted tangent bundle} and the {\em normal bundle} of the curve $C$, respectively. It is well known that the classification of rational curves by the splitting type of $f^\ast \T_{\PP^s}$  produces irreducible subvarieties of $\HH^{\rm{rat}}_{d,s}$, see \cite{verdier}, \cite{ramella}. One can look also at \cite{alzati-re} for a geometric characterization of rational curves with a given splitting of $f^\ast \T_{\PP^s}$ and \cite{iarrobino} for related results in the commutative algebra language. 

Since the early eighties of the past century a natural question about rational curves in projective spaces has been whether the subschemes of $\HH^{\rm{rat}}_{d,s}$ characterized by a given splitting of $\N_f$ are irreducible as well. This has been proved to be true for rational curves in $\PP^3$, see \cite{Eis-VdV1}, \cite{Eis-VdV2}, \cite{Gh-Sac}. The irreducibility problem has also been shown also to have a positive answer for the  general splitting type of $\N_f$, see \cite{sacchiero}, and more recently other results  related to this problem have been obtained in \cite{Ran} and \cite{bernardi}.
However the general irreducibility problem remained open.

In this paper we show that the irreducibility problem has negative answer in general, producing the first known example of a reducible Hilbert scheme of rational curves characterized by a given splitting of $\N_f$. In order to achieve this, we develop a new general method to compute the spaces of global sections $H^0\N_f(-k)$ and therefore the splitting type of $\N_f$. 
\subsection{Notations and summary of results.}
A rational curve $C\subset \PP^s$ is a curve that can be birationally parametrized by a regular map $f:\PP^1\to \PP^s$. We will always assume that $C$ is non-degenerate, i.e. not contained in any hyperplane $H\subset \PP^s$, and of degree $d>s$ with $s\geq 3$, in particular we are excluding the well known case of the rational normal curves.
Let $\mathcal{I}_C$ be the ideal sheaf of $C$ in $\PP^s$, then the normal sheaf of $C$ is the sheaf $\N_C=\underline{\rm{Hom}}_{\OO_C}(\mathcal{I}_C/\mathcal{I}^2_C,\OO_C)$. Recall also that the tangent sheaf of a noetherian scheme $X$ over $\rm{Spec}(\C)$ is defined as $\T_X=\underline{\rm{Hom}}_{\OO_X}(\Omega^1_{X/\C},\OO_X)$. Taking the differential of the parametrization map $f$ produces an exact sequence $0\to \T_{\PP^1}\stackrel{df}\to f^\ast \T_{\PP^s}\to f^\ast \N_C$. When $C$ has ordinary singularities, then $df$ is a vector bundle embedding and the sequence 
$$0\to \T_{\PP^1}\stackrel{df}\to f^\ast \T_{\PP^s}\to f^\ast \N_C\to 0$$ is exact and it identifies $f^\ast \N_C$ as the quotient bundle $f^\ast \T_{\PP^s}/df(\T_{\PP^1})$. We will denote $f^\ast \N_C=\N_f$ and call this vector bundle the normal bundle to $C$. Therefore we will assume that $C$ is irreducible and with ordinary singularities when we will be dealing with the normal bundle $\N_f$ associated to a given parametrization $f:\PP^1\to C$. 

Given a multiset of $s-1$ integers ${\overline{c}}=c_1,c_2,\ldots,c_{s-1}$, ordered so that  
$$c_1\geq c_2\geq\cdots\geq c_{s-1},$$  we will denote $\HH_{\overline{c}}$ the Hilbert scheme of irreducible degree $d$ rational curves with ordinary singularities $C\subset \PP^s$ that can be birationally parametrized by a map $f:\PP^1\to \PP^s$ such that the normal bundle $\N_f$ splits as $\N_f=\bigoplus_{i=1}^{s-1}\OO(c_i+d+2)$.

Let $U\cong\C^2$ be a two dimensional vector space and $\PP^1=\PP(U)$ its associated projective line. Let $S^dU$ be the $d$-th symmetric product of $U$. Let $\nu_d:\PP(U)\to \PP(S^dU)$ be the $d$-th Veronese embedding, and let us consider the rational normal curve $C_d=\nu_d(\PP(U))$.

Our main general result is Theorem \ref{thm:mainnow} of Section \ref{sec:mainthm}. After representing, up to projective transformations, a degree $d$ rational curve as the projection of $C_d$ from a vertex  $\PP(T)\subset \PP(S^dU)$, we prove an identification of the spaces of global sections $H^0\T_f(-d-2-k)$ and $H^0\N_f(-d-2-k)$  with the  spaces $\ker D\cap (S^kU\otimes T)\subset S^kU\otimes S^dU$ and $\ker D^2\cap (S^kU\otimes T)\subset S^kU\otimes S^dU$, respectively, where  $D$ is the first order transvectant  operator, i.e. $D=\partial_x\otimes\partial_y-\partial_y\otimes \partial_x$, with $x,y$ a basis of $U$ and $\partial_x,\partial_y$ is the dual basis, acting by derivation. By means of this result one can relate the splitting types of $\T_f$ and $\N_f$ with the position of the vertex $\PP(T)$ with respect to the rational normal curve $C_d$. 
 
 In section \ref{sec:counterexample} we introduce and discuss our example of a Hilbert scheme  $\HH_{\overline{c}}$  of rational curves $C\subset \PP^8$ of degree $d=11$ with exactly two irreducible components of dimension $98$ whose general points represent smooth rational curves, therefore providing a counterexample to the above  mentioned irreducibility problem.   
 
 In section \ref{sec:quadrics}, Theorem \ref{prop:quadricsclassif}, we give a characterization of smooth rational curves contained in rational normal scrolls in terms of the splitting type of their restricted tangent bundles and compute their normal bundles. The same theorem also shows how to construct these curves as projections of a rational normal curve.

\section{Rational curves as projections of the rational normal curve}\label{sec:notations} Given a $\C$-vector space $W$, we denote $\PP(W)$ the projective space of $1$-dimensional subspaces of $W$. More generally we denote $Gr(e+1,W)$ or $Gr(e,\PP(W))$ the Grassmannian of $(e+1)$-dimensional subspaces of $W$, or equivalently, of the $e$-dimensional linear subspaces of $\PP(W)$. If $T\subseteq W$ is a $e+1$ dimensional subspace we will denote $[T]$ or $[\PP(T)]$ its associated point in $Gr(e,\PP(W))$. Accordingly, if $w\in W$ is a non-zero vector, we will denote $[w]\in \PP(W)$ its associated point.
\subsection*{}\label{sec:dimensions}
Let $U\cong\C^2$ be a two dimensional vector space and $\PP^1=\PP(U)$ its associated projective line. Let $S^dU$ be the $d$-th symmetric product of $U$. We will denote $\nu_d:\PP^1\to \PP(S^dU)=\PP^d$ the $d$-th Veronese embedding, defined by $\nu_d(p)=[p^d]$. We will denote with $C_d=\nu_d(\PP^1)$, which is the rational normal curve given the set of pure tensors in $S^dU$. For any $b\geq 1$ we denote $Sec^{b-1}C_d$ the closure of the set of $[\tau]\in\PP(S^dU)$ such that  $\tau=p^d_1+\cdots+ p_b^d$, for $[p_i]\in C_d$ distinct points, i.e. the $(b-1)$-th  secant variety of $C_d$. 

\vskip1mm 
Let $C\subset \PP^s=\PP(V)$ be a non-degenerate rational curve of degree $d$. For the next considerations we will not need to assume that $C$ has ordinary singularities. The normalization map $\nu_C:\PP(U)\to C$ is the restriction of a map 
$f:\PP(U)\to \PP^s$ such that $f^\ast\OO_{\PP^s}(1)=\nu_C^\ast\OO_C(1)=\OO_{\PP^1}(d)$. The map $f$ is defined by an injection $f^\ast: H^0\OO_{\PP^s}(1)=V^\ast\hookrightarrow H^0\OO_{\PP^1}(d)=S^dU^\ast$ such that  $f^\ast(V^\ast)$ spans $\OO_{\PP^1}(d)$ at any point of $\PP^1$. Let us denote $$T=f^\ast(V^\ast)^\perp\subset S^dU,\quad e+1=\dim T=d-s.$$ Then one sees that the map $f^\ast$ can be identified with the dual of the map $S^dU\to S^dU/T\stackrel{\cong}\to V$. In particular, up to a linear isomorphism, we identify $\PP^s=\PP(S^dU/T)$ and the map $f$ as the composition $f=\pi_T\circ \nu_d$, where $\pi_T:\PP(S^dU)\dashrightarrow \PP(S^dU/T)$ is the projection of vertex $\PP(T)$. We want to underline the fact that for any $\overline{\psi}\in Aut(\PP^s)$, the curve $C'=\psi(C)$  is obtained by changing $f^\ast:V^\ast\to S^dU^\ast$ into  $g^\ast=f^\ast\circ \psi $ , with $\psi\in GL(V^\ast)$ a linear automorphism representing $\overline{\psi}$. Hence the space $T=f^\ast(V^\ast)^\perp$ is not affected by such a transformation. This means that one has a natural bijection between the set of orbits of maps $f:\PP^1\to\PP^s$ under the left action of $PGL(s+1)$ and the set of projection vertexes $\PP(T)$ obtained as above. 

We recall that the condition that $f^\ast(V^\ast)$ spans $\OO_{\PP1}(d)$ at any point of $\PP^1$ is equivalent to $\PP(T)\cap C_d=\emptyset $, and the fact that $f$ is birational to the image corresponds to the fact that $\PP(T)\cap Sec^1C_d$ is finite.

The discussion above shows that the Hilbert scheme $\HH^{\rm{rat}}_{d,s}$ of rational curves in $\PP^s$ is  set theoretically described as the set of images of rational maps $\pi_T\circ \nu_d$ composed with projective transformations of $\PP^s$, with the extra condition that the map $\pi_T\circ\nu_d:\PP^1\to \PP^s$ is birational to the image. More precisely, setting $\mathcal{V}$ the open subset of  $[T]\in Gr(e+1, S^dU))$ such that $\PP(T)\cap C_d=\emptyset$ and $\PP(T)\cap Sec^1T$ finite, we see that  there exists a map 
$$\mathcal{V}\times PGL(s+1)\to \HH^{\rm{rat}}_{d,s},$$ mapping $([T],\phi)\in \mathcal{V}\times PGL(s+1)$ to the curve $C=\phi(\pi_T(C_d))$ and  this map is surjective.  
\subsection{$PGL(2)$ action on the space of vertexes $\PP(T)$} Let us fix a map $f=\pi_T\circ\nu_d:\PP^1\to\PP^s$, associated to a vertex $\PP(T)$ as in the construction above. Let us consider and automorphism $\phi\in PGL(2)$. We will denote with the same letter $\phi$ also a fixed representative of the given automorphism as an element of $GL(2)$. One observes that the $d$-fold symmetric product $S^d\phi$  of the map $\phi$ acts on $S^dU$ by the action on generators $(S^d\phi)(l^d)=\phi(l)^d$  and one can define the induced action on the Grassmannian
$Gr(e+1,S^dU)$  by $[T]\mapsto [(S^d\phi)(T)]$. Now let us consider the composition 
$$f^{\phi}=f\circ\  \phi^{-1}:\PP^1\to\PP^s.$$ One has the following formula.
\begin{equation}\label{eq:PGL2action} f^{\phi}=\pi_{(S^d\phi)(T)}\circ \nu_d. \end{equation}
Indeed we know $f$ is determined by the subspace $T^\perp\subset S^dU^\ast$ and let us write $T^\perp=\langle g_0,\ldots,g_s\rangle$. Then $f^\phi$ is determined by $W= \langle g_0\circ\phi^{-1},\ldots,g_s\circ\phi^{-1}\rangle$, and by the $GL(2)$ invariance of the duality pairing $S^dU^\ast\otimes S^dU\to \C$, one immediately sees that $W=(S^d\phi)(T)^\perp\subset S^dU^\ast$. 

Above we saw that the space of maps $f:\PP^1\to \PP^s$ that birationally parametrize a non degenerate rational curve $C\subset\PP^s$ of degree $d$ is identified with $\mathcal{V}\times PGL(s+1)$, by mapping  $([T],\phi)$,  to $f=\phi\circ (\pi_T\circ\nu_d)$. 
Then we showed that the right $PGL(2)$ action on this space of maps can be identified with the left action of $PGL(2)$ on  $\mathcal{V}\times PGL(s+1)$ defined by its left action on $\mathcal{V}$.
\subsection{Irreducibility criteria and dimension formulas}
To show the irreducibility of a subscheme  $\HH_P\subseteq \HH^{\rm{rat}}_{d,s}$, defined by a geometrical property $P$ on rational curves $C\subset \PP^s$, it will be sufficient to prove the irreducibility of the subvariety $\mathcal{V}_P$ of those $[T]\in Gr(e+1,S^dU)$ such that the curve $C=\pi_T(C_d)$ satisfies property $P$. Indeed in that case $\mathcal{V}_P\times PGL(s+1)\to \HH_P$ is onto, with irreducible domain.
To be able to compute $\dim \HH_P$ from the map $\pi:\mathcal{V}_P\times PGL(s+1)\to \HH_P$ one applies the following result, which is almost obvious and very well known in the special case $\HH_P=\HH^{\rm{rat}}_{d,s}$, but that we will need in the following more general form.
\begin{prop}\label{prop:dimH} Under the notations set above,  if $\mathcal{V_P} $ is irreducible then $\HH_P$ is irreducible of dimension
$\dim \HH_P=\dim \mathcal{V}_P+\dim PGL(s+1)-3$.\end{prop}
\begin{proof} From the  above discussion it follows that the fiber over an arbitrary $[C]\in \HH_P$ is 
$$\pi^{-1}([C])=\operatorname{Orb}([T])\times \operatorname{Stab}(C),$$ with $\operatorname{Orb}([T])$ the orbit of $[T]$ under the action of $PGL(2)$ on the Grassmannian  $Gr(e+1,S^dU)$ and $\operatorname{Stab}(C)\subset PGL(s+1)$ the group of projective transformations preserving $C$.
First we consider the case when $\dim \operatorname{Orb}([T])<3=\dim PGL(2)$, i.e. when $[T]$ is fixed by some $1$-dimensional subgroup of $PGL(2)$.
The $1$-dimensional subgroups of $PGL(2)$ either fix one point $[x]\in \PP^1$ and contain the translations group acting on the basis $x,y$ as $(x,y)\mapsto (x,y+\alpha y)$, with $\alpha\in \C$, or fix two points $[x],[y]\in\PP^1$ and contain the group $(x,y)\mapsto (x,\lambda y)$, with $\lambda \in\C^\ast$. Any subspace $T\subset S^dU$ fixed by a group of the first type must contain the pure tensor $[x^d]$ and hence $[T]\not\in \mathcal{V}$. A space fixed by a subgroup of the second type is necessarily monomial, i.e. $T=\langle x^{\nu_0}y^{d-\nu_0},\ldots,x^{\nu_e}y^{d-\nu_e}\rangle$. One can see that such a space gives a point $[T]\in \mathcal{V}$, i.e. $\PP(T)\cap Sec^1C_d=\emptyset$ if and only if $d-2\geq \nu_0\geq \cdots\geq \nu_e\geq 2$, and hence it can exist if $d-3\geq e+1$. In this case one sees $\dim \operatorname{Orb}(T)=\dim PGL(2)-\dim \operatorname{Stab}(T)=2$.
\vskip2mm
Now we consider the cases when $\dim\operatorname{Stab}(C)>0$.  A classical reference for this class of curves, called the algebraic Klein Lie curves, or algebraic $W$-curves, is for example \cite{Enr-Chi}, libro V, \S 24.  Any one-dimensional subgroup of $PGL(s+1)$ whose orbits in $\PP^s$ are not lines, in a suitable coordinate system takes the form $t\mapsto\operatorname{diag}(t^{\mu_0},\ldots,t^{\mu_s})$, with $\mu_i\in \Z$ normalized and ordered so that $0= \mu_0\leq \cdots\leq\mu_s$. Its orbits $t\mapsto (\alpha_0t^{\mu_0}:\cdots: \alpha_st^{\mu_s})$ represent non-degenerate rational curves of degree $d$ if and only if the integers $\mu_i$ are distinct,  $\alpha_i\not=0$ for all $i=0,\ldots,s$ and $\mu_s=d$. Hence there exists only a finite numbers of possibilities of choices of such integers $\mu_0,\ldots,\mu_s$ for a fixed $d$, i.e. a finite number of non-degenerate degree $d$ Klein-Lie curves in $\PP^s$ up to projective equivalence. All of them can be obtained up to projective equivalence as projections $C=\pi_T(C_d)$ in the following way. For any fixed basis $x,y\in U$ consider the vertex $\PP(T)$ generated by monomials $x^{\nu_0}y^{d-\nu_0},\ldots,x^{\nu_e}y^{d-\nu_e}$, with $e+1=d-s$ and $\{\nu_0,\ldots,\nu_e\}=\{0,\ldots,d\}\setminus\{\mu_0,\ldots,\mu_s\}$, then $C=\pi_T(C_d)$ is a curve parametrized as $t\mapsto(t^{\mu_0}:\cdots:t^{\mu_s})$ with respect to the basis $(\bar{x}^{\mu_i}\bar{y}^{d-\mu_i})$ of $S^dU/T$. 
Hence we have found that non degenerate rational curves with $\dim\operatorname{Stab}(C)>0$ come from those vertexes $\PP(T)$ with $\dim\operatorname{Orb}([T])=2$ that were already analyzed above. In all those cases one has
$$\dim \pi^{-1}([C])=\dim (\operatorname{Orb}(T)\times\operatorname{Stab}(C))=2+1=3.$$ In any other case one has $\dim \operatorname{Orb}(T)=3$ and $\dim\operatorname{Stab}(C)=0$.
\end{proof}
\subsection{A classification of the projection vertexes $\PP(T)$.} Let us consider a non zero subspace $T\subseteq S^dU$, with $d\geq 2$. Let us also denote $x,y$ a basis of $U$ and $u,v$ the dual basis in $U^\ast$. Recall that $u,v$ may be identified with $\partial_x,\partial_y$ acting as linear forms on $U$, and a arbitrary element $\omega\in U^\ast$ will be written $\omega=\alpha \partial_x+\beta\partial_y$, for suitable $\alpha,\beta\in \C$. 
We define  
\begin{equation}\label{eq:deltaT}\partial T=\langle \omega(T)\ |\ \omega\in U^\ast\rangle.\end{equation}
We remark that if $U=\langle x,y\rangle$ then $\partial T=\partial_x T+\partial_y T$. One observes that in the trivial case $T=S^dU$ then  $\partial T=S^{d-1}U$. One can see that this is the only case when $\dim \partial T<\dim T$, either as an easy exercise or as a consequence of Proposition  \ref{thm:main1} below. 

We also introduce the space $\partial^{-1}T\subset S^{d+1}U$ defined in the following way.
\begin{equation}\label{eq:delta-1T}\partial^{-1} T=\bigcap_{\omega\in U^\ast}\omega^{-1}(T).\end{equation}
In this case we have $\partial^{-1}T=\partial_x^{-1}T\cap\partial_y^{-1}T$. Of course one has $\partial^{-1}S^dU=S^{d+1}U$.

 For $g\in S^{d+b}U$ we introduce the vector space
\begin{equation}\label{eq:deltaetau}\partial^b(g)=\langle\partial_x^bg,\partial_x^{b-1}\partial_yg,\ldots,\partial_y^bg\rangle\subseteq S^dU.\end{equation} Conventionally we set $\partial^b(g)=0$ if $b=-1$. 
\subsection{The numerical type of a subspace $T\subset S^dU$}\label{sec:numtype} 
We will need the following notations and results from the article \cite{alzati-re}.
\begin{definition}\label{def:secant} We will say that a proper linear space $\PP(S)\subset\PP^d$ is $C_d$-generated  if
 $\PP(S)$ is generated by its schematic intersection with $C_d$. Setting  $a+1=\dim S$, we will also say in this case that $\PP(S)$ is $(a+1)$-secant to $C_d$.  We will say that a vector subspace $S\subseteq S^d U$ is $C_d$-generated if $\PP(S)$ is $C_d$-generated. \end{definition}

\begin{notation} Given a proper subspace $T\subset S^dU$, we denote $S_T$ the smallest subspace containing the schematic intersection
$\PP(T)\cap C_d$ as a subscheme. We set $a=\dim S_T-1=\dim \PP(S_T)$, with the convention that $\dim\emptyset=-1$. \end{notation}
\begin{prop}[Theorem 1 of \cite{alzati-re}]\label{thm:main1} Let $T$ be a proper subspace $T\subseteq S^dU$. Let $S_T$ be as defined above. Then $\dim \partial S_T=\dim S_T$. Moreover if we define $r = \dim\partial T - \dim(T)$ then $r\geq 0$ and either $r=0$ and in this case one has $T=S_T$ and $T$ is $C_d$-generated, or $r\geq 1$ and there exist forms $ f_1,\ldots,f_r$, with $f_i\in \PP^{d+b_i}\setminus Sec^{b_i}C_{d+b_i}$ for $i=1,\ldots,r$,  with $b_1\geq \cdots\geq b_r\geq 0$, such that $T$ and $\partial T$ are the direct sums \[\begin{array}{l}T=S_T\oplus\partial^{b_1}(f_1)\oplus\cdots\oplus\partial^{b_r}(f_r)\\

\partial T=\partial S\oplus\partial^{b_1+1}(f_1)\oplus\cdots\oplus\partial^{b_r+1}(f_r).\end{array}\]
The $(r+1)$-uple $(a,b_1,\ldots,b_r)$ is uniquely determined from $T$. A space $T$ as above exists if and only if $a\geq -1$, $b_i\geq 0$ for all $i=1,\ldots,r$ and $a+1+\sum (b_i+2)\leq d.$ \end{prop}
\begin{definition}\label{def:type} We say that a subspace $T$ as in Proposition  \ref{thm:main1} has {\em numerical type} $(a,b_1,\ldots,b_r)$. If $S_T=0$, i.e. $\PP(T)\cap C_d=\emptyset$ then $a=-1$ and we will say $T$ has type $(b_1,\ldots,b_r)$. \end{definition}
Let us also recall the following other result from \cite{alzati-re}.
\begin{prop}[Proposition 5 of \cite{alzati-re}]\label{prop:partial-1T} Assume that $T\subseteq S^dU$ has type $(a,b_1,\ldots,b_r)$ so that it has a decomposition $$T=S_T\oplus\bigoplus_{i=1}^r \partial^{b_i}(f_i)$$ satisfying the requirements of Proposition \ref{thm:main1}. Then $\partial^{-1}(S_T)=S_{\partial^{-1}T}$, $\dim \partial^{-1}(S_T)=\dim S_T=a+1$ and there exists a decomposition $$\partial^{-1}T=\partial^{-1}S_T\oplus \bigoplus_{i:b_i\geq 1} \partial^{b_i-1}(f_i).$$ In particular $\partial^{-1}T$ has type $(a,b_1-1,\ldots,b_{r_1}-1)$, with $r_1=\max(i:\ b_i\geq 1)$.\end{prop}
\subsection{The splitting type of the restricted tangent bundle of rational curves.}
The main result of \cite{alzati-re} about the splitting type of the restricted tangent bundle $f^\ast \T_{\PP^s}$, that we will write shortly as $\T_f$, of a parametrized rational curve $f:\PP^1\to \PP^s$ is the following.
\begin{prop}[Theorem 3 of \cite{alzati-re}]\label{prop:tangristr} Assume that $f:\PP^1\to\PP^s$ is obtained by the projecting the rational normal curve $C_d$ from a vertex $\PP(T)$ with $T$ of type $(b_1,\ldots,b_r)$. Then $r\leq s$ and the splitting type of $\T_f$ is $$\T_f=\OO_{\PP^1}(b_1+d+2)\oplus\cdots \oplus\OO_{\PP^1}(b_r+d+2)\oplus\OO_{\PP^1}^{s-r}(d+1).$$ \end{prop}
We also recall the restricted Euler sequence $$0\to \OO_{\PP^1}\to (S^dU/T)\otimes \OO_{\PP^1}(d)\to  \T_f\to 0,$$ from which one gets
$\deg \T_f=(s+1)d$.
\section{Review of some $SL(U)$ invariant operators.}\label{sec:tensors}
In this section we will review some well known invariant operators between spaces of tensors  on $U$ or $U^\ast$, for convenience of the reader and for later reference. Invariance will mean $GL(U)$ or $SL(U)$ invariance.  
\subsection{The duality pairing} It is the natural pairing $S^dU^\ast\otimes S^dU\to \C$ that identifies either of the two spaces as the dual of the other. It may be defined considering any element of $S^dU^\ast$ as a differential operator on $S^dU$. More precisely if $x,y\in U$ and $u,v\in U^\ast$ are dual bases, then one has the formula $$f(u,v)\in S^dU^\ast,\ l=\lambda x+\mu y\in U\Rightarrow f(l^d)=d!f(\lambda,\mu).$$
\subsection{General contractions} The contraction maps $$S^kU^\ast\otimes S^bU\to S^{b-k}U,$$ defined for any $0\leq k\leq b$ or the analogous maps interchanging $U$ and $U^\ast$ can be interpreted similarly as above by letting the tensors in $S^kU^\ast$ act on $S^bU$ as differential operators.  The following formulas are straightforward consequences of the definition of the action of $f\in S^kU^\ast$ as a differential operator.
\begin{eqnarray}\label{eq:contraction}f(l^b)&=&{b\choose k}f(l^k)l^{b-k}.\\
\label{eq:contrderiv} f(\eta(g))&=&(\eta f)(g),\ \forall f\in S^kU^\ast,\ \forall \eta\in U^\ast,\ \forall g\in S^{b+1}U.
\end{eqnarray}
\subsection{The multiplication maps} These are the maps
 $m:S^iU\otimes S^jU\to S^{i+j}U$, or the same with $U^\ast$ in the place of $U$, defined on pure generators by $m(l^i\otimes h^j)= l^ih^j$.
\subsection{The polarization maps} These are maps $p_{k}:S^{d+k}U\to S^kU\otimes S^dU$ proportional to  duals of the multiplication maps $m:S^kU^\ast\otimes S^dU^\ast\to S^{d+k}U^\ast$, with proportionality factor determined so that $m(p_{k}(f))=f$ for any $f\in S^{d+k}U$.  For this reason, the polarization maps are always injective. The maps $p_{k}$ are uniquely defined by 
$$p_{k}(l^{d+k})=l^k\otimes l^d.$$
One has the following well-known closed formula for $p_k$ in terms of a fixed basis $x,y$ for $U$:
\begin{equation}\label{eq:pkformula}p_k(f)=\frac{(\deg{f}-k)!}{\deg{f}!}\sum_{i=0}^k{k\choose i}x^{k-i}y^i\otimes\partial_x^{k-i}\partial_y^i(f).\end{equation} 
\subsection{The multiplication by  $\xi=x\otimes y- y\otimes x$}  This is a $SL(U)$-invariant element of $U\otimes U$, that indeed generates the  irreducible subrepresentation of $GL(U)$ given by $U\wedge U\subset U\otimes U$. The multiplication by $\xi$ acts in the following way:
$$\xi:S^{i-1}U\otimes S^{j-1}U\to S^iU\otimes S^jU.$$
Observe that for any $k\leq d$ one has the direct sum decomposition
\begin{equation}\label{eq:pieri} S^kU\otimes S^dU=p_{k}(S^{d+k}U)\oplus\xi p_{k-1}(S^{d+k-2}U)\oplus\cdots\oplus \xi^kp_{0}(S^{d-2k}U).
\end{equation}
Here we set $S^iU=0$ if $i<0$. The decomposition above is indeed equal to the {\em Pieri decomposition} of $S^kU\otimes S^dU$ as a $GL(U)$ representation, for which we refer to \cite{fulton-harris}. Note that grouping the terms in (\ref{eq:pieri}) in a suitable way one obtains \begin{eqnarray}S^kU\otimes S^dU&=&p_{k}(S^{d+k}U)\oplus\xi (S^{k-1}U\otimes S^{d-1}U),\label{eq:decompD}
\\
S^kU\otimes S^dU&=&p_{k}(S^{d+k}U)\oplus\xi p_{k-1}(S^{d+k-2}U)\oplus \xi^2 (S^{k-2}U\otimes S^{d-2}U)\label{eq:decompD^2}.
\end{eqnarray}

 \subsection{The operator $D=D_{x,y}=\partial_x\otimes \partial_y-\partial_y\otimes\partial_x$} It is classically know as the {\em first order transvectant}, see e.g. \cite{olver}, Definition 5.2. If $(x',y')=(x,y)A$ is a new basis for $U$, then the operator $D$ transforms as $D_{x',y'}=(\det A)^{-1}D_{x,y}$, see \cite{olver}, formula (5.3). In particular $D$ is invariant with respect to the $SL(U)$ representation on $U^\ast\otimes U^\ast$. In this article we will consider the following actions of $D$ as a differential operator:
 $$D:S^kU\otimes S^dU\to S^{k-1}U\otimes S^{d-1}U.$$ 
  The operator $D$ satisfies the following property.
 \begin{lm}\label{lm:Dprop} For any $\tau\in S^{k-1}U\otimes S^{d-1}U$ one has $$D(\xi \tau)=(d+k)\tau+\xi D(\tau).$$
 Moreover one has $D(p_k(f))=0$ for any $f\in S^{d+k}U$. \end{lm}
 We omit the proof, that can be achieved by a direct computation reducing oneself to the case $\tau=x^{k-1}\otimes y^{d-1}$ by linearity and $SL(2)$ invariance.
 One consequence of the Lemma above is the following.
\begin{cor}\label{cor:kerD^2} For  any $d,k\geq 1$ (resp. $d,k\geq 2$), the next two sequences are exact 
\[ \begin{array}{l}0\to p_k(S^{d+k}U)\to S^kU\otimes S^dU \stackrel{D}\to S^{k-1}U\otimes S^{d-1}U\to0\\

 0\to p_k(S^{d+k}U)\oplus \xi p_{k-1}(S^{d+k-2}U)\to S^kU\otimes S^dU \stackrel{D^2}\to S^{k-2}U\otimes S^{d-2}U\to 0.\end{array}\]
\end{cor}
\begin{proof} We start with the first sequence. The fact that the sequence is a complex is the second statement of Lemma \ref{lm:Dprop}. By the first statement of Lemma \ref{lm:Dprop} and by (\ref{eq:pieri}) and (\ref{eq:decompD}), the operator $D$ maps the subspace $\xi (S^{k-1}U\otimes S^{d-1}U)$ of the space $S^kU\otimes S^dU=p_{k}(S^{d+k}U)\oplus\xi (S^{k-1}U\otimes S^{d-1}U)$ onto $S^{k-1}U\otimes S^{d-1}U$. The exactness in the middle also follows from the decomposition (\ref{eq:decompD}). The proof of the exactness of the second sequence is very similar. One first shows that $$D^2(p_k(S^{d+k}U)\oplus \xi p_{k-1}(S^{d+k-2}U))=0$$ by applying Lemma \ref{lm:Dprop} twice. Then the exactness follows from (\ref{eq:pieri}) and (\ref{eq:decompD^2}) in a similar way as above.   \end{proof}

On a different vein, one can use the operator $D^2$ to produce the invariant map
\begin{equation}\label{eq:maxrkmap}S^kU\otimes S^bU^\ast\stackrel{D^2}\longrightarrow S^{k-2}U\otimes S^{b+2}U^\ast.\end{equation}
In this map, the tensor $D^2=\partial_x^2\otimes \partial_y^2-2\partial_x\partial_y\otimes \partial_x\partial_y+\partial_y^2\otimes\partial_x^2$ acts by contraction on the $S^kU$ components and by multiplication on the $S^bU^\ast$ component.
Later we will need the following result.
\begin{prop}\label{prop:maxrk} The map (\ref{eq:maxrkmap}) has maximal rank for any $b\geq 0$ and $k\geq 2$.\end{prop}
\begin{proof} We use the identification $\phi:U^\ast\to U$ that maps $\alpha \partial_x+\beta \partial_y\mapsto -\beta x+\alpha y$. Note  that $\phi$ is $SL(2)$-invariant, as $\phi\wedge \phi$ maps $\partial_x\wedge \partial_y\mapsto y\wedge(- x)=x\wedge y$. Then for any $i,j\geq 0$ the map $1\otimes S^j(\phi):S^iU\otimes S^jU^\ast\to S^iU\otimes S^jU$ is a isomorphism. We can rewrite the map (\ref{eq:maxrkmap}) in terms of these identifications as follows
$$ S^kU\otimes S^bU\stackrel{\delta^2}\longrightarrow S^{k-2}U\otimes S^{b+2}U,$$ with $\delta^2=\partial_x^2\otimes x^2+2\partial_x\partial_y\otimes xy+\partial_y^2\otimes y^2=(\partial_x\otimes x+\partial_y\otimes y)^2$, acting as before by contraction on $S^kU$ and by multiplication on $S^bU$. Now the fact that $\delta^2$ has maximal rank is consequence of the following  similar but more general result. \end{proof}
\begin{lm}For any $(n+1)$-dimensional $\C$-vector space $V=\langle x_0,\ldots,x_n\rangle$ and any $k\geq a$ and $b\geq 0$, setting $\delta=(\partial_{x_0}\otimes x_0+\cdots+\partial_{x_n}\otimes x_n)$, the map
\begin{equation}\label{eq:delta}S^kV\otimes S^bV\stackrel{\delta^a}\longrightarrow S^{k-a}V\otimes S^{b+a}V\end{equation} has maximal rank. 
\end{lm}
\begin{proof}[Sketch of proof] This result should be rather well known, however we sketch a proof for convenience of the reader. We use the invariance of $\delta$ and the Pieri decompositions of $S^kV\otimes S^bV$ and $S^{k-a}V\otimes S^{b+a}V$ as $SL(V)$-modules. As it is well known 
\begin{equation}\label{eq:SkSb}S^kV\otimes S^bV=\bigoplus_{i=0}^{\min(k,b)}S_{(k+b-i,i)}V,\end{equation} where $S_{(k+b-i,i)}V$ is the $SL(V)$-irreducible tensor space resulting by applying to $V$ the Schur functor associated to the Young diagram with two rows of lengths $k+b-i$ and $i$, respectively.  One has the similar decomposition 
\begin{equation}\label{eq:SkSba}S^{k-a}V\otimes S^{b+a}V=\bigoplus_{i=0}^{\min(k-a,a+b)}S_{(k+b-i,i)}V.\end{equation}
Note that if $b\leq k-a$, then all the summands $S_{(k+b-i,i)}V$ appearing in (\ref{eq:SkSb}) appear also in (\ref{eq:SkSba}) and, on the other hand, if $b\geq k-a$, then all the summands in (\ref{eq:SkSba}) appear in (\ref{eq:SkSb}). Then the proof is complete if one shows that for any summand appering in both the formulas above, the composition 
$$S_{(k+b-i,i)}V\hookrightarrow S^kV\otimes S^bV\stackrel{\delta^a}\longrightarrow S^{k-a}V\otimes S^{b+a}V\twoheadrightarrow S_{(k+b-i,i)}V$$ is non-zero and hence an isomorphism. It is well known that the first invariant inclusion identifies $S_{(k+b-i,i)}V$ as the subspace of $S^kV\otimes S^bV$ generated by tensors of the form $\xi_1\ldots\xi_if$, where $\xi_j$ are tensors of the form $x_h\otimes x_k-x_k\otimes x_h$ and $f\in p_{k-i}(S^{k+b-2i}V)\subset S^{k-i}V\otimes S^{b-i}V$. Then one observes the fundamental fact that $\delta(x_h\otimes x_k-x_k\otimes x_h)=0$. Since $\delta$ is a derivation on the commutative ring $S^\bullet V\otimes S^\bullet V$, one deduces that $\delta$ commutes with $x_h\otimes x_k-x_k\otimes x_h$ and hence $\delta^a(\xi_1\ldots\xi_if)=\xi_1\ldots\xi_i\delta^a(f)$. Then one concludes by the observation that $f=p_{k-i}(g)$ and one can easily check that $\delta^a(f)=\delta^{a}(p_{k-i}(g))=p_{k-i-a}(g)$, up to some non-zero rational factor. Hence the map $\delta^a$ is non-zero when restricted to $S_{(k+b-i,i)}V$. \end{proof}

\subsection{The invariant embeddings $\psi_k:U\otimes S^{d+k-1}U\to S^kU\otimes S^dU$}
We define these maps as the compositions 
\[\begin{CD}U\otimes S^{d+k-1}U@>1\otimes p_k>> U\otimes S^kU\otimes S^{d-1}U@>\widetilde{m}>>S^kU\otimes S^dU,\end{CD}\] where $\widetilde{m}$ is the multiplication of the first and the third tensor components of $U\otimes S^kU\otimes S^{d-1}U$. The maps $\psi_k$ are obviously $SL(U)$-invariant. We will show that the maps $\psi_k$ are invariant embeddings for any $k\geq 1$. 
\begin{prop}\label{prop:psik} For any $d\geq 2$ and $k\geq 1$ the map $\psi_k$ is injective  and  $$\psi_k(U\otimes S^{d+k-1}U)=\ker (D^2: S^kU\otimes S^dU\to S^{k-2}U\otimes S^{d-2}U),$$ where the map above is set to be the zero map in the case $k=1$.\end{prop}
\begin{proof}  We use the decomposition $U\otimes S^{d+k-1}U=p_1(S^{d+k}U)\oplus \xi S^{d+k-2}U$, which is a particular case of (\ref{eq:decompD}). Since the two summands are  irreducible representations of $SL(U)$ and the map $\psi_k$ is $SL(U)$ invariant, to show the injectivity of $\psi_k$ it will be sufficient to show that $\psi_k$ is non-zero on the summands $p_1(S^{d+k}U)$ and $\xi S^{d+k-2}U$. We will achieve that by computing $\psi_k$ on some special elements of these summands. 

For $l\otimes l^{d+k-1}\in p_1(S^{d+k}U)$ we see that 
\begin{eqnarray*}\psi_k(l\otimes l^{d+k-1})&=&\widetilde{m}((1\otimes p_k)(l\otimes l^{d+k-1}))\\
&=&\widetilde{m}(l\otimes l^k\otimes l^{d-1})\\
&=&l^k\otimes l^d\in p_k(S^{d+k}U)\subset S^kU\otimes S^dU.\end{eqnarray*}
Now let us consider the element $\xi x^{d+k-2}=x\otimes x^{d+k-2}y-y\otimes x^{d+k-1}\in \xi S^{d+k-2}U$. We compute separately $\psi_k(x\otimes x^{d+k-2}y)$ and $\psi_k(y\otimes x^{d+k-1})$. One finds easily $$\psi_k(y\otimes x^{d+k-1})=x^k\otimes x^{d-1}y.$$ 

From formula (\ref{eq:pkformula}) one has 
\begin{eqnarray*}p_k(x^{d+k-2}y)&=&\frac{(d-1)!}{(d+k-1)!}(x^k\otimes\partial_x^k(x^{d+k-2}y)+kx^{k-1}y\otimes\partial_x^{k-1}\partial_y(x^{d+k-2}y))\\
&=&\Small{\frac{(d-1)!}{(d+k-1)!}}\left(\frac{(d+k-2)!}{(d-2)!}x^k\otimes x^{d-2}y+kx^{k-1}y\otimes \frac{(d+k-2)!}{(d-1)!}x^{d-1}\right)\\
&=&\frac{1}{d+k-1}((d-1)x^k\otimes x^{d-2}y+kx^{k-1}y\otimes x^{d-1}).
\end{eqnarray*}
Hence one obtains
$$\psi_k(x\otimes x^{d+k-2}y)=\frac{1}{d+k-1}((d-1)x^k\otimes x^{d-1}y+kx^{k-1}y\otimes x^{d}).$$ Then we find
\begin{eqnarray*}\psi_k(\xi x^{d+k-1})&=&\psi_k(x\otimes x^{d+k-2}y)-\psi_k(y\otimes x^{d+k-1})\\
&=&\frac{1}{d+k-1}((d-1)x^k\otimes x^{d-1}y+kx^{k-1}y\otimes x^{d})\\&-&\frac{1}{d+k-1}((d+k-1)x^k\otimes x^{d-1}y)\\
&=&\frac{k}{d+k-1}(x^{k-1}y\otimes x^{d}-x^k\otimes x^{d-1}y)\\
&=&-\frac{k}{d+k-1}\xi (x^{k-1}\otimes x^{d-1})\in \xi p_{k-1}(S^{d+k-2}U).
\end{eqnarray*}
The calculations made above show that $\psi_k$ restricts to a non zero $SL(U)$-invariant map on $p_1(S^{d+k}U)$ and $\xi S^{d+k-2}U$, in particular by the $SL(U)$ irreducibility of these spaces, one gets \begin{eqnarray*}\psi_k(p_1(S^{d+k}U))&=& p_{k}(S^{d+k}U)\\ \psi_k(\xi S^{d+k-2}U)&=& \xi p_{k-1}(S^{d+k-2}U),\end{eqnarray*} proving the global injectivity of $\psi_k$.  Moreover, applying Corollary \ref{cor:kerD^2}, one has 
$$\psi_k(U\otimes S^{d+k-1}U)=p_{k}(S^{d+k}U)\oplus \xi p_{k-1}(S^{d+k-2}U)=\ker D^2.$$
\end{proof}
\section{A new setup for computing the cohomology of $\N_f$}\label{sec:mainthm}
From now on we will assume $f:\PP^1\to \PP^s$ parametrizes a rational curve with ordinary singularities and that $f=\pi_T\circ \nu_d$, so the parametrized curve arises as projection of the rational normal curve $C_d$ from a vertex $\PP(T)$.
Let us recall the operator $$D^2:S^kU\otimes S^dU\to S^{k-2}U\otimes S^{d-2}U$$ discussed in Section \ref{sec:tensors}.
 We state the main theorem of this article, whose proof will be given at the end of this section. 
  \begin{thm}\label{thm:mainnow} For any $k\geq 1$ one has \begin{eqnarray*}h^0\T_f(-d-2-k)&=&\dim (\ker D\cap(S^kU\otimes T))\\
  h^0\N_f(-d-2-k)&=&\dim (\ker D^2\cap(S^kU\otimes T)).\end{eqnarray*} \end{thm}
\subsection{Euler sequence and  its consequences}  
Let $C\subset\PP^s$ be a degree $d$ rational curve with ordinary singularities. As in the notation above we assume there is a parametrization map $f:\PP^1\to \PP^s$ obtained projecting the rational normal curve $C_d$ from a vertex $\PP(T)\subset \PP(S^dU)$. Since $f=\pi_T\circ \nu_d$, we have $\PP^s=\PP(S^dU/T)$. Note also that the natural inclusion $(S^dU/T)^\ast\subset S^dU^\ast$  identifies $(S^dU/T)^\ast=T^\perp$. Hence we can set
$$\dim T=e+1,\ \dim T^\perp=s+1=d-e.$$
We have a commutative diagram
\[\begin{CD} &&\OO_{\PP^1}@>=>>\OO_{\PP^1}&\\
&&@VVV @VVV\\
0@>>> U\otimes\OO_{\PP^1}(1)@>J(f)>> (T^\perp)^\ast\otimes\OO_{\PP^1}(d)@>>> \N_f@>>> 0\\
&&@VVV @VVV @VV\rm{id}V\\
0@>>>\T_{\PP^1}@>df>> \T_f@>>> \N_f@>>>0.\end{CD}\]
Indeed, if the map $f:\PP^1\to \PP((T^\perp)^\ast)=\PP^s$ is given in coordinates by $$f(u:v)=(g_0(u,v):\cdots:g_s(u,u)),$$ with $g_i(u,v)\in S^dU^\ast$, then the map $J(f):U\otimes \OO_{\PP^1}(1)\to (T^\perp)^\ast\otimes\OO_{\PP^1}(d)$ in the diagram above is given fiberwise by the differentials $df|_{(u,v)}:T_{(u,v)}(C\PP^1)\to T_{f(u,v)}(C\PP^s)$ of the map among the associated affine cones $f:C\PP^1\to C\PP^s$. Hence it has associated matrix
\[J(f)=\left(\begin{array}{ll}\partial_ug_0(u,v)&\partial_vg_0(u,v)\\
\vdots&\vdots\\ \partial_ug_s(u,v)&\partial_vg_s(u,v)\end{array}\right).\]
Let us consider the exact sequence 
\begin{equation}\label{eq:normal}0\to U\otimes\OO_{\PP^1}(1)\to (T^\perp)^\ast\otimes\OO_{\PP^1}(d)\to \N_f\to 0.\end{equation} From this sequence we get $$\deg \N_f(-d-1)=-(d-e)+2d=d+e.$$
Writing, as in the introduction, 
\begin{equation}\label{eq:normalsplitting}\N_f=\OO_{\PP^1}(c_1+d+2)\oplus\cdots\oplus \OO_{\PP^1}(c_{s-1}+d+2)\end{equation} with $c_1\geq\cdots\geq c_{s-1}$, we see that \begin{equation}\label{eq:c_i}
\sum_{i=1}^{s-1}(c_i+1)=d+e,\quad 
\sum_{i=1}^{s-1}c_i=2(e+1).\end{equation}

Taking the cohomology exact sequence from (\ref{eq:normal}) we obtain, for any $k\geq d+1$,
\begin{equation}\label{eq:normalcohom}H^0\N_f(-k)\hookrightarrow U\otimes H^1\OO_{\PP^1}(1-k)\to (T^\perp)^\ast\otimes H^1\OO_{\PP^1}(d-k)\twoheadrightarrow H^1\N_f(-k).\end{equation}
If $k=d+1$ one obtains $H^0\N_f(-d-1)\cong U\otimes H^1\OO_{\PP^1}(-d)$. Let us now consider the cases $k\geq d+2$. We have $T^\perp=\langle g_0,\ldots,g_s\rangle$ and let us denote $g_0^\ast,\ldots,g_s^\ast$ the dual basis to the $g_i$'s in $(T^\perp)^\ast=S^dU/T$. Recall that if we write $U^\ast=\langle u,v\rangle$, with $u,v$ the dual basis of $x,y\in U$, then   the first nonzero map in (\ref{eq:normal}) is defined by $x\otimes l\mapsto \sum_i g_i^\ast\otimes l \partial_ug_i$ and $y\otimes l'\mapsto \sum_i g_i^\ast\otimes l'\partial_vg_i$, for any local sections $l,\ l'$ of $\OO_{\PP^1}(1)$.

As it is well known, by Serre duality one can identify the spaces $H^1\OO_{\PP^1}(1-k)$ and $H^1\OO_{\PP^1}(d-k)$ appearing in the exact sequence (\ref{eq:normalcohom}) with  $(H^0\OO_{\PP^1}(k-3))^\ast=S^{k-3}U$ and $(H^0\OO_{\PP^1}(k-d-2))^\ast=S^{k-d-2}U$, respectively. Moreover it is well known that any sheaf map  $\OO_{\PP^1}(1-k)\stackrel{\sigma}\to\OO_{\PP^1}(d-k)$, associated to a global section $\sigma\in H^0\OO_{\PP^1}(d-1)=S^{d-1}U^\ast$, induces a map between the cohomology spaces $H^1\OO_{\PP^1}(1-k)\stackrel{\sigma}\to H^1\OO_{\PP^1}(d-k)$ that, under the identifications above, can be written as the linear map $S^{k-3}U\stackrel{\sigma}\to S^{k-d-2}U$ defined by letting $\sigma$ act as a differential operator on $S^{k-3}U$. In our case the sheaf map $U\otimes\OO_{\PP^1}(1-k)\to (T^\perp)^\ast\otimes\OO_{\PP^1}(d-k)$ arising from (\ref{eq:normal}), after the identifications $U\cong\C^2$ and $T^\perp\cong\C^{s+1}$ by means of the mentioned bases $x,y$ and $g_0,\ldots,g_{s}$, can be seen as a sheaf map $\OO_{\PP^1}^2(1-k)\to \OO_{\PP^1}^{s+1}(d-k)$, whose components have the form $\OO_{\PP^1}(1-k)\stackrel{\partial_ug_i}\to\OO_{\PP^1}(d-k)$ and $\OO_{\PP^1}(1-k)\stackrel{\partial_vg_i}\to\OO_{\PP^1}(d-k)$. The induced maps on the $H^1$ cohomology spaces are therefore $\partial_ug_i: S^{k-3}U\to S^{k-2-d}U$ and $\partial_vg_i: S^{k-3}U\to S^{k-2-d}U$,  acting as  differential operators of order $d-1$. 

From the discussion above it follows that 
one can compute $H^0\N_f(-k)$ as the kernel of the linear map 
\begin{equation}\label{eq:H0Naskernel} U\otimes S^{k-3}U\to (T^\perp)^\ast\otimes S^{k-d-2}U \end{equation} defined by $x\otimes f\mapsto\sum_i g_i^\ast\otimes (\partial_u g_i)(f)$ and  $y\otimes f'\mapsto\sum_i g_i^\ast\otimes (\partial_v g_i)(f')$, where $\partial_ug_i,\partial_vg_i: S^{k-3}U\to S^{k-2-d}U$  act as  differential operators of order $d-1$. Let us compute the kernel  $H^0\N_f(-k)$ of the linear map (\ref{eq:H0Naskernel}). 

The space  $H^0\N_f(-k)$, seen as a subspace of $U\otimes S^{k-3}U$,  is the space of tensors $x\otimes f_0+y\otimes f_1\in U\otimes S^{k-3}U$ such that $(\partial_u g_i)(f_0)+(\partial_v g_i)(f_1)=0\in S^{k-d-2}U$ for all $i=0,\ldots,s$. This is equivalent to impose that
 $f_0(P\partial_ug)+f_1(P\partial_vg)=0$ for any $g\in T^\perp$ and any $P\in S^{k-d-2}U^\ast$. This is equivalent to say that $$P(f_0)(\partial_ug)+P(f_1)(\partial_vg)=0$$ for any $P\in S^{k-d-2}U^\ast$ and any $g\in T^\perp$. 
By applying the version of formula (\ref{eq:contrderiv}) with the roles of $U$ and $U^\ast$ interchanged and recalling that the elements $x,y\in U$ act as  $\partial_u,\partial_v$ on $\C[u,v]$, respectively, one sees that for any $\phi\in S^{d-1}U$ and any $g\in S^dU^\ast$ one has  $\phi(\partial_ug)=(x\phi)(g)$ and similarly $\phi(\partial_vg)=(y\phi)(g)$. Hence we can rewrite the last displayed equation in the following form
$$(xP(f_1)+yP(f_2))(g)=0\quad \forall\ g\in T^\perp,\ \forall P\in S^{k-d-2}U^\ast, $$ which means
\begin{equation}\label{eq:T'_kcond} xP(f_1)+yP(f_2)\in T,\ \forall P\in S^{k-d-2}U^\ast.  \end{equation}
 
 \begin{notation} The calculations made above hold for any $k\geq d+2$. We find it convenient, from now on, to reset $k\leftarrow k-d-2$. 
 Accordingly, we denote, for any $k\geq 0$,
   $$T_k=\{x\otimes f_0+y\otimes f_1\in U\otimes S^{d+k-1}U\ |\ x P(f_0)+yP(f_1)\in T,\ \forall\ P\in S^kU^\ast\}.$$ \end{notation}
   Hence we can resume the discussion above in the following result.
  \begin{prop}\label{thm:relTpTp_ast} Under the notation above, we have the following relation for any $k\geq d+2$.
 \begin{equation}\label{eq:h0NT}H^0\N_f(-d-2-k)=T_k.\end{equation}\end{prop}
 The following proposition collects some facts that will be needed later, as well as some first applications of the result above.
 \begin{prop}\label{prop:seconddiff} Assume that $\N_f$ has a splitting of the form (\ref{eq:normalsplitting}). Then the following facts hold. 
 \begin{enumerate} 
 \item $h^0\N_f(-d-k-2)=\sum_{i:c_i\geq k}(c_i-k+1)$ for any $k\in \mathbb{Z}$. 
 \item Setting $f(-k)=h^0\N_f(-d-k-2)$ for any $k\in \mathbb{Z}$, one has $$\#\{i\ |\ c_i=k\}=\Delta^2f(-k)=f(-k)-2f(-k-1)+f(-k-2).$$
 \item $\sum_{i=1}^{s-1}(c_i+1)=d+e=d+\dim\PP(T)$. 
 \item $\sum_{i=1}^{s-1}c_i=2(e+1)=2\dim T$.  
 \item $c_{s-1}\geq 0$. 
 \end{enumerate}\end{prop}
 \begin{proof} Facts (1) and (2) are easy and well known. The relations (3) and (4) coincide with formulas (\ref{eq:c_i}) and therefore have already been proven. 
 
 From Proposition \ref{thm:relTpTp_ast} we have the identification
 $$H^0\N_f(-d-2)=\{x\otimes f_1+y\otimes f_2\in U\otimes S^{d-1}U\ |\ xf_1+yf_2\in T\}$$ and therefore we see
 that \begin{equation}\label{eq:k0}H^0N_f(-d-2)\cong m^{-1}(T)\subset U\otimes S^{d-1}U,\end{equation} 
 where $m$ is the multiplication map $m:U\otimes S^{d-1}U\to S^dU$. Now, the kernel of $m$ is given by tensors of the form $x\otimes yh-y\otimes xh$, with arbitrary $h\in S^{d-2}U$.  Then one has 
 \begin{equation}\label{eq:k=0}h^0\N_f(-d-2)=\dim m^{-1}(T)=d-1+\dim T=d+e.\end{equation}
 On the other hand, by (\ref{eq:c_i}) we know $$d+e=h^0\N_f(-d-2)=\sum_{i:c_i\geq 0}(c_i+1)\geq \sum_{i=1}^{s-1}(c_i+1)=d+e$$
  This implies that $c_1\geq\cdots\geq c_{s-1}\geq -1$.
 We will also need to know the value of $h^0\N_f(-d-1)$. This is obtained from the exact sequence (\ref{eq:normal}), from which it easily follows $H^0\N_f(-d-1)\cong U\otimes H^1\OO_{\PP^1}(-d)$ and hence $h^0\N_f(-d-1)=2(d-1)$. Now, applying (2) for $k=-1$ and using (3), (4) and the above calculation of $f(1)=h^0\N_f(-d-1)$, we see that $\#\{i\ |\ c_i=-1\}=2(d-1)-2(d+e)+2(e+1)=0$, which completes the proof of (5).
   \end{proof}

  \subsection{Completion of the proof of Theorem \ref{thm:mainnow}}
  \begin{proof}[Proof of Theorem \ref{thm:mainnow}] We start with the part of the statement about $\T_f$. At the beginning of p. 1334,  section 6.2 of \cite{alzati-re} we showed the equality
  $$h^0\T_f(-d-2-k)=\dim\partial^{-k}T.$$ Moreover from Corollary \ref{cor:kerD^2} we know that $p_k(S^{d+k}U)=\ker D\subset S^kU\otimes S^dU$.  Then one finds that \begin{eqnarray*}\ker D\cap (S^kU\otimes T)&=&p_k(S^{k+d}U)\cap (S^kU\otimes T)\\
  &=&p_k(\{f\in S^{d+k}U\ |\ \partial_x^{k-i}\partial_y^i(f)\in T,\ \forall\ i=0,\ldots,k\})\\
  &\cong &\partial^{-k}T.\end{eqnarray*} 
  Hence we find the equality $h^0\T_f(-d-2-k)=\dim(\ker D\cap (S^kU\otimes U))$.
  
   Now we prove the statement about $\N_f$. 
 By proposition \ref{thm:relTpTp_ast} we know 
  $$H^0\N_f(-d-2-k)=T_k$$ with 
  $T_k\subseteq U\otimes S^{d+k-1}U$ the subspace consisting of those elements $x\otimes f_0+y\otimes f_1$ such that $xP(f_0)+yP(f_1)\in T$ for any $P\in S^kU^\ast$. This is equivalent to the condition
  $$x\partial_x^{k-i}\partial_y^i(f_0)+y\partial_x^{k-i}\partial_y^i(f_1)\in T,\ \forall\ i=0,\ldots,k.$$
  Recall that by the formula (\ref{eq:pkformula}) one has 
  \begin{eqnarray*} \psi_k(x\otimes f_0+y\otimes f_1)&=&\widetilde{m}(x\otimes p_k(f_0)+y\otimes p_k(f_1))\\
  &=&\operatorname{const}\cdot \sum_{i=1}^k{k\choose i}x^{k-i}y^i\otimes (x\partial_x^{k-i}\partial_y^i(f_0)+y\partial_x^{k-i}\partial_y^i(f_1)).\end{eqnarray*} 
  Therefore, by the definition of $T_k$, we have 
  \begin{eqnarray*}x\otimes f_0+y\otimes f_1\in T_k&\iff& x\partial_x^{k-i}\partial_y^i(f_0)+y\partial_x^{k-i}\partial_y^i(f_1)\in T\ \forall\ i=0,\ldots,k\\
  &\iff&\psi_k(x\otimes f_0+y\otimes f_1)\in S^kU\otimes T.\end{eqnarray*}
   On the other hand, by Proposition \ref{prop:psik}, one has $\psi_k(x\otimes f_0+y\otimes f_1)\in \operatorname{Im}(\psi_k)=\ker D^2$ and
   $\psi_k$ is injective for $k\geq 1$.
  Hence for any $k\geq 1$ one has  $$ H^0\N_f(-d-2-k)\cong T_k\stackrel{\psi_k}\cong \ker D^2\cap (S^kU\otimes T).$$
   \end{proof}
\section{Some general consequences of Theorem \ref{thm:mainnow}.}\label{section:genconsequences}
\subsection{$h^0\N_f(-d-2-k)$ for $k=0,1,2$.}
\begin{prop}\label{prop:firstk} 
 The spaces $H^0\N_f(-d-2-k)$ have the following dimensions, for $k=0,1,2$. 
\[\begin{array}{lllll}   k=0, & & h^0\N_f(-d-2)&=&d-1+\dim T.\\
   k=1, & &h^0\N_f(-d-3)&=&2\dim T.\\
   k=2, & & h^0\N_f(-d-4)&=&3\dim T-\dim \partial^2T.
   \end{array}\]
   \end{prop}
\begin{proof} The case $k=0$ is the formula (\ref{eq:k=0}) and it has already been discussed. 

The case $k=1$ is a consequence of the degree of $\N_f$ and it was already established by the formulas (\ref{eq:c_i}), but it also follows from the fact that $D^2=0$ on the space $U\otimes S^dU$ and therefore, by Theorem \ref{thm:mainnow}, one has $H^0\N_f(-d-3)\cong U\otimes T$.

Finally, for $k=2$, by Theorem \ref{thm:mainnow} we have to compute $\dim((S^2U\otimes T)\cap\ker D^2)=\dim\ker D^2|_{S^2U\otimes T}$. Note that $\dim (S^2U\otimes T)=3\dim T$ and hence the claim on $h^0\N_f(-d-4)$ follows if we show that $D^2(S^2U\otimes T)=\partial^2T$. We know  that $$D^2((ax^2+bxy+cy^2)\otimes \tau)=2a\tau_{xx}-2b\tau_{xy}+2c\tau_{yy}.$$ By choosing $\tau \in T$ and $a,b,c$ appropriately, one sees that $\tau_{xx},\tau_{xy},\tau_{yy}\in D^2(S^2U\otimes T)$ and since these elements generate $\partial^2T$ one obtains $\partial^2T\subseteq D^2(S^2U\otimes T)$. The converse inclusion is obvious. \end{proof}
\begin{cor}\label{cor:d+2} The number of summands equal to $\OO_{\PP^1}(d+2)$ in the splitting type (\ref{eq:normalsplitting}) of $\N_f$ is equal to $d-1-\dim\partial^2T$. \end{cor}
\begin{proof} This follows immediately from Proposition \ref{prop:seconddiff}, (5) applied to $k=0$ and using the dimensions computed in Proposition \ref{prop:firstk}.\end{proof}
 
\subsection{Some general results on $h^0\N_f(-d-2-k)$ with $k\geq 3$.}
The computation of kernels and  images of the maps $$D^2:S^kU\otimes T\to S^{k-2}U\otimes S^{d-2}U$$ for $k\geq 3$ may be not easy for a arbitrary $T$. Sometimes one can reduce this computation to the case of subspaces of smaller dimension. This is possible by means of the following easy lemma.
   \begin{lm}\label{lm:D2reduction} Assume that for a given decomposition $T=T_1\oplus T_2$ one also has $\partial^2T=\partial^2T_1\oplus\partial^2T_2$. Then  for any $k\geq 2$ the map $D^2: S^kU\otimes T\to S^{k-2}U\otimes S^{d-2}U$ is the direct sum of its restrictions to $S^kU\otimes T_i$, for $i=1,2$. In particular its rank is the sum of the ranks of the two restrictions.\end{lm}
   \begin{proof} Immediate, since the image of $\operatorname{res}(D^2): S^kU\otimes T_i\to S^{k-2}U\otimes S^{d-2}U$ is contained in $S^{k-2}U\otimes \partial^2T_i$ for $i=1,2$.\end{proof}
   From the Lemma above one deduces the following result.
   \begin{prop}\label{prop:nosovrapp} Assume that $T=\partial^{b_1}(f_1)\oplus\cdots\oplus\partial^{b_r}(f_r)$ of type $(b_1,\ldots,b_r)$ and that $\partial T$ has type $(b_1+1,\ldots, b_r+1)$. Let us denote $$D_i^2: S^kU\otimes \partial^{b_i}(f_i)\to S^{k-2}U\otimes \partial^{b_i+2}(f_i)$$ the restriction of $D^2$, for any $i=1,\ldots,r$. Then the maps $D_i^2$ have maximal rank for any $i=1,\ldots,r$ and the rank of $D^2:S^kU\otimes T\to S^{k-2}U\otimes S^{d-2}U$ is the sum of their ranks. \end{prop}
\begin{proof} In view of lemma \ref{lm:D2reduction} we only need to show that $D^2_i$ has maximal rank for any $i=1,\ldots,r$.  Note that by Proposition \ref{thm:main1} the assumption that the type of $\partial T$ is $(b_1+1,\ldots,b_r+1)$ in particular implies $\dim \partial^{b_i+2}(f_i)=b_i+3$ for all $i$, hence one has a isomorphism $S^{b_i+2}U^\ast\to  \partial^{b_i+2}(f_i)$ defined by $\Omega\mapsto \Omega(f_i)$ for any $\Omega\in S^{b_i+2}U^\ast$. Recall also that since $T$ has type $(b_1,\ldots,b_r)$ one knows $\dim \partial^{b_i}(f_i)=b_i+1$ and hence the isomorphism $S^{b_i}U^\ast\to \partial^{b_i}(f_i)$ defined in the same way as above. Under these isomorphisms the maps $D_i^2$ are identified  with the map (\ref{eq:maxrkmap}) with $b=b_i$ and hence, by Proposition  \ref{prop:maxrk}, they have maximal rank. \end{proof}
As an application of the result above, we compute the normal bundles of rational curves obtained from vertexes $T$ of the most special type, i.e. $T=\partial^e(g)$ with $g\in\PP^(S^{d+e})\setminus Sec^eC_{d+e}$.
\begin{prop}\label{prop:specialnb} If the curve $C\subset\PP^s$ is obtained from a vertex $T$ of numerical type $(e)$, that is $T=\partial^e(g)$ with $g\in\PP(S^{d+e})\setminus Sec^eC_{d+e}$, then $$\N_f= \OO_{\PP^1}^2(d+e+3)\oplus \OO_{\PP^1}^{d-e-4}(d+2).$$\end{prop}
 \begin{proof} One can apply Proposition \ref{prop:nosovrapp} and find
$$h^0\N_f(-d-2-k)=\max(0, (k+1)(e+1)-(k-1)(e+3))=\max (0,2e+4-2k).$$ Setting,  as in Proposition \ref{prop:seconddiff}, $f(-k)=h^0\N_f(-d-2-k)$ for $k\geq 0$, we see that the sequence $f(-k)$ is the following one:
$$ d+e, 2e+2,2e,\ldots, 2,0\cdots$$ Its second difference is the following:
$$d-e-4,0,\ldots,0,2, 0\cdots$$ where the last $2$ appears at the place $k=e+1$. Hence, by Proposition \ref{prop:seconddiff}, one has $(c_1,\ldots,c_{s-1})=(e+1,e+1,0,\ldots,0)$, with $s-1=d-e-2$. By formula (\ref{eq:normalsplitting}), we obtain the stated splitting type of $\N_f$.
\end{proof} 
\section{Example of a reducible Hilbert scheme of rational curves with fixed normal bundle: $\HH_{\overline{c}}$ with ${\overline{c}}=(2,2,1,1,0,0,0)$}\label{sec:counterexample} 
This section is dedicated to the construction of the first known example, to our knowledge, of a reducible Hilbert scheme of rational curves with a given splitting type of the normal bundle.

As in the introduction, we will denote $\HH_{\overline{c}}$ the Hilbert scheme of degree $d$ irreducible, non degenerate rational curves in $\PP^s$, with ordinary singularities and with normal bundle with splitting type $\bigoplus\OO_{\PP^1}(c_i+d+2)$.
We will consider the case when ${\overline{c}}=(2,2,1,1,0,0,0)$, therefore we have $s-1=7$, moreover from $\sum(c_i+1)=13=d+e$ and $\sum c_i=6=2(e+1)$ we get $e=2$ and $d=11$, i.e. we are dealing with rational curves of degree $11$ in $\PP^8$. Precisely, we are dealing with parametrized curves of degree $11$ in $\PP^8$ with splitting type of the normal bundle given by
$$\N_f=\OO_{\PP^1}^2(15)\oplus \OO_{\PP^1}^2(14)\oplus \OO_{\PP^1}^3(13).$$ These curves are obtained, up to a projective transformation in $\PP^8$,  as projection of the rational curve $C_{11}=\nu_{11}(\PP^1)\subseteq\PP(S^{11}U)$ from a $2$-dimensional vertex $\PP(T)$, so that $$\dim T=e+1=3.$$
 We recall that the knowledge of the $s-1$-uple $(c_1,\ldots,c_{s-1})$ is equivalent to the knowledge of the dimensions of the spaces $H^0\N_f(-d-2-k)= T_k$. In our case these dimensions are the following:
\begin{eqnarray*} \dim T_0&=&\sum_{i:c_i\geq 
0}(c_i+1)=13\\
\dim T_1&=&\sum_{i:c_i\geq 1}c_i=6\\
\dim T_2&=&\sum_{i:c_i\geq 2}(c_i-1)=2\\
\dim T_3&=&\sum_{i:c_i\geq 3}(c_i-2)=0\\
\dim T_k&=&0\ \forall\ k\geq 3.\end{eqnarray*}
We also recall that $T_k\cong \ker (D^2:S^kU\otimes T\to S^{k-2}U\otimes \partial^2T)$ for all $k\geq 1$.
\noindent
Since the vertex $\PP(T)$ must not intersect $C_{11}$, we have only three possibilities for the numerical type of $T$, namely the type $(2)$, the type $(1,0)$ and the type $(0,0,0)$. We can immediately rule out the type $(2)$ by the following argument. By Proposition \ref{prop:firstk} one has \begin{equation}\label{eq:dimd2T}\dim \partial^2T=\dim S^2U\otimes T-\dim T_2=7.\end{equation} If $T$ is of type $(2)$, then $T=\partial^2(f)$ for some polynomial $f\in S^{13}U$ and hence $\partial^2T=\partial^4(g)$, which has dimension at most $5$. Therefore we are left with the possibility that $T$ has type $(1,0)$ or type $(0,0,0)$. 
\subsection{Curves from spaces $T$ of type $(1,0)$}
We will show that from  a general vertex $T$ of type $(1,0)$ we always obtain a curve with splitting of the normal bundle corresponding to $\overline{c}=(2,2,1,1,0,0,0)$. Recall such a vertex has the form 
$$T=\partial(f)\oplus \langle g\rangle,$$ with sufficiently general $f\in\PP(S^{12}U)$ and $g\in \PP(S^{11}U)$, this last polynomial being determined by $T$ up to an element of $\partial(f)$. Hence the dimension of the space of such $T$'s is given by $\dim\PP(S^{12}U)+\dim \PP(S^{11}U/\partial(f))=12+9=21$.  The same conclusion can be reached by means of the dimension formula provided by Theorem 2 of \cite{alzati-re}.
 
Now we know that a general $T\subset S^{11}U$ of type $(1,0)$ has $\partial T$ of type $(2,1)$. This may be shown starting from a particular $T$, for example $T=\langle x^3y^8, x^4y^7, x^7y^4\rangle=\partial(x^4y^8)\oplus(x^7y^4)$, from which we get the direct sum decompositions  $\partial T=\partial^2(x^4y^8)\oplus\partial(x^7y^4)$ and $\partial^2 T=\partial^3(x^4y^8)\oplus\partial^2(x^7y^4)$. Then one can extend the result to a general $T$ of type $(1,0)$ by lower semicontinuity of $\dim\partial^2T$.   Hence for a general $T$ of type $(1,0)$ we find $\dim\partial^2T=\dim\partial T+2=7$, as required by (\ref{eq:dimd2T}). In particular one obtains $\partial^2T=\partial^3(f)\oplus\partial^2(g)$ and for any $k\geq 2$ the map $D^2:S^kU\otimes T\to S^{k-2}U\otimes \partial^2T$ can be written as the direct sum of the maps \[
\begin{array}{l} D^2:S^kU\otimes \partial(f)\to S^{k-2}U\otimes \partial^3(f)\\

D^2:S^kU\otimes (g)\to S^{k-2}U\otimes \partial^2(g).\end{array}\]
By construction one has $\dim\partial(f)=2$, $\dim\partial^3(f)=4$, $\dim(g)=1$ and $\dim\partial^2(g)=3$,  hence one can identify $S^iU^\ast \cong \partial^i(f)$ for $i=1,3$ and $S^jU^\ast\cong \partial^j(g)$ for $j=0,2$. By means of these identifications the maps above become 
\[\begin{array}{l} D^2:S^kU\otimes U^\ast\to S^{k-2}U\otimes S^3U^\ast\\

D^2:S^kU\otimes S^0U^\ast\to S^{k-2}U\otimes S^2U^\ast,\end{array}\]
where now $D^2$ operates as in Proposition \ref{prop:maxrk}.  Hence the maps have maximal rank.
 For $k=3$ the map $D^2:S^3U\otimes \partial(f)\to U\otimes \partial^3(f)$ has domain of dimension $8$ and codomain of dimension $8$, hence is an isomorphism. The map
$D^2:S^3U\otimes (g)\to U\otimes \partial^2(g)$ has domain of dimension $4$ and codomain of dimension $6$, hence it is injective. In conclusion we obtain $T_3=0$, and hence also $T_k=0$ for all $k\geq 3$. So we get the dimensions of the spaces $T_k$ that correspond to $\overline{c}=(2,2,1,1,0,0,0)$. By Proposition \ref{prop:dimH} we have obtained a irreducible subscheme of $\HH_{\overline{c}}$ of dimension $21+\dim PGL(9)-\dim PGL(2)=98$. 

We observe that the general curve in the subscheme of $\HH_{\overline{c}}$ just defined is a smooth rational curve. Indeed this is equivalent to showing that a general $\PP(T)$ with $T$ of type $(1,0)$ as above does not intersect $Sec^1C_{11}$. Let us fix $[g]\in\PP^{11}\setminus Sec^1C_{11}$, then the dimension of the cone over $Sec^1C_{11}$ with vertex $[g]$,  defined as the join $J=J ([g], Sec^1C_{11})$, is $\dim J=4$. Let us define $$J'=\{[f']\in \PP(S^{12}U)\ |\ \exists\ \omega\in U^\ast :\ [\omega (f')]\in J\}.$$ Then one finds $\dim J'\leq 6$, indeed $J'=\bigcup_{q\in J, \omega\in \PP(U^\ast)}\PP(\omega^{-1}(q))$,   and therefore there exists $[f]\in \PP^{12}\setminus J'$. Then one can conclude that for a general $T=\partial(f)\oplus \langle g\rangle$ one has $$\PP(T)\cap Sec^1C_d=\emptyset.$$
\subsection{Curves from spaces $T$ of type $(0,0,0)$}
Unlike the previous case of $T$ of type $(1,0)$, it will not be true that a general $T\subseteq S^{11}U$ of type $(0,0,0)$ can produce a rational curve in $\HH_{\overline{c}}$. Instead, we will show that the space of all $T$ of type $(0,0,0)$ whose general element produces curves in $\HH_{\overline{c}}$ is a proper irreducible subvariety of the space of all $T$ of type $(0,0,0)$. 

Now we have $\dim \partial T=\dim T+3=6$ and remind that to obtain a curve in $\HH_{\overline{c}}$ one must have $\dim\partial^2T=7$. Hence, under the notations of Proposition \ref{thm:main1}, the space $\partial T$ has type $(a,b_1)$ with $\dim \partial T=a+1+b_1+1=6$, i.e. $(a,b_1)=(a,4-a)$. 
\paragraph{\bf Case $a=-1$.} One has $a=-1$ if and only if $\PP(\partial T)$ does not intersect $C_{10}\subset \PP(S^{10}U)$, so we see that $\partial T$ has type $(b_1)=(5)$, i.e. $\partial T=\partial^5(g)$ for some $[g]\not\in Sec^5C_{15}\subset \PP^{15}$ and hence $\partial^2T=\partial^6(g)$ has dimension $7$, as required.  

We compute the dimension of the variety of the spaces $T$ under consideration. We observe that for a fixed general 
$[g]\in\PP(S^{15}U)$, any sufficiently general $T\subseteq \partial^{-1}T=\partial^4(g)$ will have type $(0,0,0)$ and $\partial T=\partial^5(g)$. One can first show the claim for a special couple $T,g$, for example $g=x^8y^7$ and $T=\langle x^4y^7, x^6y^5,x^8y^3\rangle$. Then the result holds for a general $T,g$ by semicontinuity, more precisely by upper semicontinuity of $\dim \partial^{-1}T$, which is equal to $0$ if and only if $T$ has type $(0,0,0)$, by Proposition \ref{prop:partial-1T}. Hence we can find spaces $T$ meeting our requirements in a dense open set of $Gr(3,\partial^4(g))$, whose dimension is $\dim Gr(3,\partial^4(g))=6$.  Moreover, since a general $T\subset \partial^4(g)$ constructed as above is such that $\partial T=\partial^5(g)$, then $\langle g\rangle=\partial^{-5}(\partial T)$ is uniquely determined by $T$. Hence the final count of parameters for spaces $T$ as above is the following: $$\dim \PP(S^{15}U)+\dim Gr(3,5)=15+6=21.$$ 
\paragraph{\bf Case $a\geq 0$} By Proposition \ref{thm:main1} a general $T$ of type  $(a,4-a)$ has the form $$\partial T=\langle p_0^{10},\ldots,p_a^{10}\rangle\oplus\partial^{4-a}(g)$$ for a suitable $[g]\not\in Sec^{4-a}C_{14-a}\subset \PP(S^{14-a}U)$. Note that the $C_{10}$-generated part of $\partial T$ is uniquely determined by $\partial T$ and hence by $T$, i.e. the points $p_0,\ldots,p_a$ are uniquely determined. On the other hand $g$ is determined only modulo $W=\langle p_0^{14-a},\ldots,p_a^{14-a}\rangle$. We have  $$T\subseteq \partial^{-1}\partial T=\langle p_0^{10},\ldots,p_a^{10}\rangle\oplus\partial^{3-a}(g),$$ 
which is again a space of dimension $5$, uniquely determined by $T$. However now $[g]\in \PP(S^{14-a}U/W)$, which gives us $13-2a$ parameters. Hence a similar dimension count as above provides us with a number of parameters equal to $13-2a+a+1+\dim Gr(3,5)=20-a$ parameters. So in the case $a\geq 0$ we find a variety of vertexes $\PP(T)$ of smaller dimension then in the case $a=-1$. Since we are looking for components of $\HH_{\overline{c}}$ of maximal dimensions, we will be satisfied if we get one such component from the case $a=-1$.
\vskip2mm
So we reduced ourselves to show that a general $T$ of type $(0,0,0)$ with $\partial T$ of type $(5)$ produces a curve in $\HH_{\overline{c}}$. Note that from the known data $d=11$, $\dim T=3$ $\dim\partial^2T=7$ we already have $\dim T_0=d+\dim T=13$, $\dim T_1=2\dim T=6$, $\dim T_2=3\dim T-\dim\partial^2T=2$. From the caracterization of $\overline{c}=(2,2,1,1,0,0,0)$ in terms of the dimensions of the spaces $T_k$, we will get a curve in $\HH_{\overline{c}}$ from the vertex $T$ if and only if one has $\dim T_3=0$. By semicontinuity, if we show this for a special $T$ of type $(0,0,0)$ and $\partial T$ of type $(5)$, then the same will hold for the general such $T$. We take the same example as above.
$$g=x^8y^7,\ T=\langle x^8y^3, x^6y^5,x^4y^7\rangle.$$
\begin{notation} To simplify calculations we denote with $[h]$ any fixed non-zero rational multiple of the polynomial $h$. Similarly $[h]+[g]$ will denote a fixed linear combination of $h,g$ with non-zero rational coefficients.\end{notation}
We compute $ T_3$ as the kernel of $D^2:S^3U\otimes T\to U\otimes \partial^2T$. In particular we will get $T_3=0$ if we show that the image of that map has dimension $12$.  Recalling that $D^2=\partial_x^2\otimes\partial_y^2-2\partial_x\partial_y\otimes \partial_x\partial_y+\partial_y^2\otimes\partial_x^2$, we see the following:
\[\begin{array}{l}D^2(\langle x^3,x^2y,xy^2,y^3\rangle\otimes \langle x^8y^3 \rangle)=\\
 \langle [x\otimes x^8 y],[y\otimes x^8 y]+[x\otimes x^7y^2],[y\otimes x^7y^2]+[x\otimes x^6y^3],[y\otimes x^6y^3]\rangle,\\
 \\
 D^2(\langle x^3,x^2y,xy^2,y^3\rangle\otimes \langle x^6y^5 \rangle)=\\
 \langle [x\otimes x^6 y^3],[y\otimes x^6 y^3]+[x\otimes x^5y^4],[y\otimes x^5y^4]+[x\otimes x^4y^5],[y\otimes x^4y^5]\rangle,\\
 \\
 D^2(\langle x^3,x^2y,xy^2,y^3\rangle\otimes \langle x^4y^7 \rangle)=\\
\langle [x\otimes x^4 y^5],[y\otimes x^4 y^5]+[x\otimes x^3y^6],[y\otimes x^3y^6]+[x\otimes x^2y^7],[y\otimes x^2y^7]\rangle .
  \end{array}\] 
  The space $D^2(S^3\otimes T)$ is generated by the $12$ elements shown on the right hand sides of the equalities above.
  After taking suitable linear combinations of them, they are reduced to the following set of generators
  
 $$\begin{array}{l} \mbox{$[x\otimes x^8 y],\ [y\otimes x^8 y]+[x\otimes x^7y^2],\ [y\otimes x^7y^2],\ [y\otimes x^6y^3]$}\\
 \mbox{$[x\otimes x^6 y^3],\ [x\otimes x^5y^4],\ [y\otimes x^5y^4],\ [y\otimes x^4y^5]$}\\
 \mbox{$[x\otimes x^4 y^5],\ [x\otimes x^3y^6],\ [y\otimes x^3y^6]+[x\otimes x^2y^7],\ [y\otimes x^2y^7]$}.
 \end{array}$$
 After this simplification, one can easily see that the $12$ generators are linearly independent. This completes the proof that $T_3=0$.
 
 Finally we observe that in the given example of $T=\langle x^8y^3, x^6y^5,x^4y^7\rangle$ one has $$T^\perp=\langle u^{11},u^{10}v,u^9v^2,u^7v^4,u^5v^6,u^3v^8,u^2v^9,uv^{10},v^{11}\rangle$$ and since the elements of given basis of $T^\perp$ serve also as the components of a parametrization map $f=\pi_T\circ\nu_d:\PP^1\to\PP^s$, one easily sees that the parametrized curve is smooth. Hence also the general curve in the same component of $\HH_{\overline{c}}$ is smooth.
 \begin{conclusion} We have found that for $\overline{c}=(2,2,1,1,0,0,0)$ the Hilbert scheme $\HH_{\overline{c}}$ is the union of two irreducible components, each of dimension equal to $21+\dim PGL(9)-\dim PGL(2)=98$, by Proposition \ref{prop:dimH}. One component has general point representing a smooth rational curve constructed from a general vertex $T$ of type $(1,0)$ with $\partial T$ of type $(2,1)$. The other component has general point representing smooth rational curves constructed from a general vertex $T$ of type $(0,0,0)$ with $\partial T$ of type $(5)$. We  also observe that, by Proposition  \ref{prop:tangristr}, the restricted tangent bundles are the following (setting $d=11$):
\begin{eqnarray*} f^\ast\T_{\PP^s}&=&\OO_{\PP^1}(d+3)\oplus \OO_{\PP^1}(d+2)\oplus \OO_{\PP^1}^6(d+1),\  \mbox{for $T$ of type }(1,0);  \\
f^\ast\T_{\PP^s}&=&\OO_{\PP^1}^3(d+2)\oplus \OO_{\PP^1}^5(d+1),\  \mbox{for $T$ of type }(0,0,0); 
\end{eqnarray*}
 On the other hand, for any  $[C]\in \HH_{\overline{c}}$ one has
 $$\N_f=\OO_{\PP^1}^2(d+4)\oplus \OO_{\PP^1}^2(d+3)\oplus \OO_{\PP^1}^3(d+2).$$\end{conclusion}
 \begin{rmk} One may note that the decomposition type given above has the form $\N_f=\mathcal{F}\oplus \OO_{\PP^1}^3(d+2)$ with $\mathcal{F}=\OO_{\PP^1}^2(d+4)\oplus\OO_{\PP^1}^2(d+3)$ of {\em almost balanced type}, and hence $\N_f$ has the most general possible type among the  vector bundles on $\PP^1$ of the same rank and degree and with the summand $\OO_{\PP^1}^3(d+2)$.  Therefore the  same counterexample discussed in this section gives also the following.
 \begin{example} The variety parametrizing the rational curves of degree $d=11$ in $\PP^8$ with normal bundle $\N_f$ with $3$ summands of degree $d+2=13$ is reducible. \end{example}
 This is actually a counterexample to Theorem 4.8 of \cite{bernardi}. It seems that in the preparatory results leading to Theorem 4.8, especially Lemma 4.3, the author has overlooked his own more detailed treatment of the same results given in his Ph.D. thesis \cite{bernardiThesis}, where more restrictive hypotheses are given. In \cite{bernardiThesis} Theorem 4.8 of \cite{bernardi} is stated as Theorem 3.4.16, which in turn is deduced from Theorem 3.3.9 and Theorem 3.4.10. Our counterexample corresponds to the case $n=11$, $d=8$, $k=3$, $r=2$ and $\rho_r^{n,k}=3$ in the author's notation, and it is not covered by Theorems 3.3.9 and 3.4.10 of \cite{bernardiThesis}.
 \end{rmk}
 \section{smooth rational curves in rational normal scrolls}\label{sec:quadrics}
In this section we will  characterize smooth rational curves contained in rational normal scroll surfaces in terms of the splitting type of their restricted tangent bundles $\T_f$ and we will also compute the splitting type of their normal bundle $\N_f$. Our main result can be viewed as a generalization of  Propositions 5 and 6 of \cite{Eis-VdV1}, where the authors characterized smooth rational curves contained in a smooth quadric in $\PP^3$ by their restricted tangent bundle and computed their normal bundle. The general purpose of this section is to illustrate the idea that especially the splitting type of $\T_f$ may have a deep impact on the extrinsic geometry of the curve $C\subset\PP^s$.

\begin{notation} Following the notations of \cite{Hartshorne}, chapter 2, Section 7, we denote $\mathbf{P}(\mathcal{E})$ the projective bundle associated to a vector bundle $\mathcal{E}$ on $\PP^1$ or rank $t\geq 1$. Recall that an epimorphism of  vector bundles $\C^{s+1}\otimes\OO_{\PP^1}\to \mathcal{E}$ defines a regular map $g:\mathbf{P}(\mathcal{E})\to\PP^s$ such that, denoting $H$ the pullback of an hyperplane of $\PP^s$,  one has $\deg H^{t-1}=\deg\mathcal{E}=\deg \wedge^{t}\mathcal{E}$.  If the map $g:\mathbf{P}(\mathcal{E})\to\PP^s$ is birational to the image, then, setting $S=\operatorname{Im}(g)$, one finds $\deg S=\deg \mathcal{E}$. \end{notation}
Let $C\subset \PP^s$ be a smooth non-degenerate rational curve of degree $d$, biregularly parametrized by a map $f:\PP^1\to\PP^s$, which, as discussed in preceding sections, we can assume of the form $f=\pi_T\circ \nu_d$ up to a projective transformation of $\PP^s$. As usual we will set $\dim T=e+1$ and $s=d-e-1$. Throughout this section we will assume $s\geq 3$ and $d\geq s+1$, i.e. $T\not=0$. First we study a sufficient condition for $C$ to be smooth.
\begin{lm}\label{lm:smoothtypee} Let $T=\partial^e(g)$ be a vertex of type $(e)$. Then the curve $C=\pi_T(C_d)$ is smooth if and only if $g\in \PP(S^{d+e}U)\setminus Sec^{e+1}C_{d+e}$. \end{lm}
 \begin{proof} Our strategy of proof will be to show that when $T$ has type $(e)$, the curve $C$ is smooth if and only if $\partial T$ has type $(e+1)$. Indeed, by Proposition \ref{thm:main1} one sees $\partial T=\partial^{e+1}(g)$ being of type $(e+1)$ is equivalent to $[g]\not\in Sec^{e+1}C_{d+e}$. Note that the point $[g]\in\PP(S^{d+e}U)$ such that $\partial T=\partial^{e+1}(g)$ has type $(e+1)$ is unique, since one sees that $\langle g\rangle=\partial^{-e-1}(\partial T)$, by iteratively applying Proposition \ref{prop:partial-1T}.
 
 The condition that $C$ is smooth is given by $\PP(T)\cap Sec^1C_d=\emptyset$.  Observe that $T$ being of type $(e)$ in particular implies $\PP(T)\cap C_d=\emptyset$ and $\dim\PP(\partial T)=\dim \PP(T)+1$. Hence the space $\PP(\partial T)$, that a priori is the join $\PP(\langle \omega(T)\ |\ [\omega]\in \PP(U^\ast)\rangle)$, in this case is also the union $\PP(\partial T)=\bigcup_{\omega\in U^\ast}\PP(\omega(T))$.  Then one has $\PP(\partial T)\cap C_{d-1}\not=\emptyset$ if and only if there exists $\omega \in U^\ast$ and $l\in U$ such that $[l^{d-1}]\in \PP(\omega(T))$. Setting $\langle m\rangle =\omega^\perp$, this is equivalent to say that in $\PP(T)$ either there exists an element of the form $[\lambda l^d+\mu m^d]$ if $[m]\not=[l]$, or an element of the form $[l^{d-1}n]$ if $[m]=[l]$. This is equivalent to the condition $\PP(T)\cap Sec^1C_d\not=\emptyset$, i.e. to $C$ not smooth. 
 
 Therefore we have shown that $C$ is smooth if and only if  $\PP(\partial T)\cap C_{d-1}=\emptyset$, i.e. $S_{\partial T}=0$, under the notations of Proposition \ref{thm:main1}.  Moreover, for $T=\partial^e(g)$,  one has $\partial T=\partial^{e+1}(g)$ and $\partial^2 T=\partial^{e+2}(g)$, hence $\dim\partial^2T-\dim\partial T\leq 1$.  Then, by Proposition \ref{thm:main1} applied to the space $\partial T$, we see that $C$ is smooth if and only if $\partial T$ has type $(e+1)$. \end{proof}
\begin{rmk} Note that the open set $\PP(S^{d+e}U)\setminus Sec^{e+1}C_{d+e}$ is non empty and of dimension $d+e=2d-s-1$ if and only if $\dim Sec^{e+1}C_{d+e}=2e+3\leq d+e-1$, which is true, as we are assuming $s=d-e-1\geq 3$.\end{rmk}
Now we can state and prove the main result of this section.
\begin{thm}\label{prop:quadricsclassif} Let us assume that $C$ is a non-degenerate irreducible smooth rational curve of degree  $d\geq s+1$ and parametrization map $f=\pi_T\circ \nu_d:\PP^1\to C\subset \PP^s$.  Then the following facts are equivalent.
 \begin{enumerate} \item The vertex $T$ is of type $(e)$, i.e. $T=\partial^e(g)$ with $[g]\in \PP(S^{d+e}U)\setminus Sec^{e+1}C_{d+e}$. 
 \item $\T_f= \OO_{\PP^1}(d+2+e)\oplus\OO_{\PP^1}^{s-1}(d+1)$.
  \item $C$ is contained in a smooth rational normal scroll $S\cong\mathbf{P}(\mathcal{E})\subset \PP^s$, with $\mathcal{E}=\OO_{\PP^1}(\alpha)\oplus\OO_{\PP^1}(\beta)$, $\alpha,\beta>0$, $\alpha+\beta=s-1$. 
\end{enumerate}
Moreover, under any of the conditions above, the following facts also hold.
\begin{enumerate}
\item[\it (i)] The rational normal scroll containing $C$ is uniquely determined by $C$. 
\item[\it (ii)] The normal bundle $\N_f$ has splitting type
$\N_f\cong \OO_{\PP^1}^2(d+e+3)\oplus \OO_{\PP^1}^{s-3}(d+2)$.
\end{enumerate}
  \end{thm}
 \begin{proof} 
 (1)$\iff$(2). By Proposition \ref{prop:tangristr} one sees that $T$ has type $(e)$, i.e. $T=\partial^e(g)$ with $[g]\not\in Sec^eC_{d+e}$ if and only if 
 $\T_f=\OO_{\PP^1}(d+2+e)\oplus\OO_{\PP^1}^{s-1}(d+1)$. Since we are assuming $C$ smooth, by Lemma \ref{lm:smoothtypee} one actually has  $[g]\not\in Sec^{e+1}C_{d+e}$.
 \vskip2mm
 \paragraph{(2)$\Rightarrow$(3)}
 We set $V=T^\perp$ and recall the restricted Euler sequence appearing in the second column of the diagram of section 4.1:
 $$0\to  \OO_{\PP^1}\to V^\ast\otimes  \OO_{\PP^1}(d)\to \T_f\to 0.$$ From this sequence and the existence of the sub-line bundle $\OO_{\PP^1}(d+2+e)\to \T_f$ we deduce a commutative diagram with exact rows and columns
\[\begin{CD}0@>>>  \OO_{\PP^1}(-d)@>>> \mathcal{E}^\ast@>>>  \OO_{\PP^1}(e+2)@>>>0\\
&&  @VV\cong V @VVV @VVV\\
0@>>>  \OO_{\PP^1}(-d)@>>> V^\ast\otimes  \OO_{\PP^1}@>>>\T_f(-d)@>>> 0\\
&&&& @VVV @VVV\\
&&&& \OO_{\PP^1}^{s-1}(1)@>\cong>>\OO_{\PP^1}^{s-1}(1)\end{CD}\]
where $\mathcal{E}^\ast$ is defined as the preimage of $\OO_{\PP^1}(e+2)$ in $V^\ast\otimes \OO_{\PP^1}$.  Dually, we get a exact sequence $0\to\OO_{\PP^1}^{s-1}(-1)\to V\otimes \OO_{\PP^1}\longrightarrow \mathcal{E}\to 0$.  It immediately follows that $\mathcal{E}$ has splitting type $\mathcal{E}\cong \OO_{\PP^1}(\alpha)\oplus\OO_{\PP^1}(\beta)$ with $\alpha,\beta\geq 0$ and  $\alpha+\beta=s-1$. Moreover, the sheaf map $V\otimes \OO_{\PP^1}\to\OO_{\PP^1}(d)$ that is naturally associated to $f$ is the composition of the sheaf epimorphisms  $V\otimes \OO_{\PP^1}\to \mathcal{E}\to \OO_{\PP^1}(d)$.  Let us denote $Y=\mathbf{P}(\mathcal{E})$. Then the sheaf epimorphism $V\otimes \OO_{\PP^1}\to \mathcal{E}$ provides a map $Y\to \PP^s$, whose image $S$ is a ruled surface of minimal degree $s-1$, and the existence of the factorization $V\otimes \OO_{\PP^1}\to \mathcal{E}\to \OO_{\PP^1}(d)$ shows that the curve $C$ is contained in $S$ as the image of a section $\widetilde{C}$ of the $\PP^1$ bundle $Y\to\PP^1$. We have only to show that $\alpha,\beta>0$. Indeed if for example $\alpha=0$ and $\beta=s-1$, then $S$ is a  cone  over a rational normal curve in $\PP^{s-1}$,  more precisely the map $Y\to S$ contracts the unique curve  $C_0$ of $Y$ with $C_0^2=1-s$ to the vertex of the cone $S$. In this case the section $\widetilde{C}\subset Y$ has divisor class $\widetilde{C}\equiv C_0+dF$, with $F$ a fibre of $Y\to\PP^1$, and intersection number $\widetilde{C}\cdot C_0=d+1-s\geq 2$ for $d\geq s+1$. Hence $C$ cannot be smooth for $d\geq s+1$. This argument excludes the case of the cone, therefore $\mathcal{E}=\OO_{\PP^1}(\alpha)\oplus\OO_{\PP^1}(\beta)$, with $\alpha+\beta=s-1$ and $\alpha,\beta>0$. In this case one also sees that the map $Y\to\PP^1$ is an embedding, i.e. $Y\cong S$, so $S$ is a smooth rational normal scroll.
 \vskip2mm
\paragraph{$(3)\Rightarrow (2)$} Assume that $C\subset S\subset \PP^s$, with $S$ a smooth rational normal scroll. In particular $S$ is isomorphic to a rational ruled surface $\mathbf{P}(\mathcal{E})$, embedded in $\PP^s$ by means of a surjection of vector bundles $V\otimes \OO_{\PP^1}\to \mathcal{E}$. The fact that $\deg S=s-1$ is equivalent to $\deg \mathcal{E}=s-1$. The fact that $C\subset S\cong \mathbf{P}(\mathcal{E})$ is a section of the projection map $\mathbf{P}(\mathcal{E})\to \PP^1$ implies the existence of a sheaf epimorphism $\mathcal{E}\to\OO_{\PP^1}(d)$ such that the epimorphism $V\otimes\OO_{\PP^1}\to\OO_{\PP^1}(d)$ associated to the embedding $C\subset\PP^s$ factorizes as $V\otimes \OO_{\PP^1}\to \mathcal{E}\to \OO_{\PP^1}(d)$. 
Setting $\mathcal{L}=\ker(\mathcal{E}\to\OO_{\PP^1}(d))$, we see that $\mathcal{L}\cong \OO_{\PP^1}(s-1-d)=\OO_{\PP^1}(-e-2)$. Now we can dualize all the sheaf morphisms that we have introduced so far, obtaining a diagram of the form
\begin{equation}\label{eq:diagram3-->2}\begin{CD}0@>>>  \OO_{\PP^1}(-d)@>>> \mathcal{E}^\ast@>>>  \OO_{\PP^1}(e+2)@>>>0\\
&&  @VV\cong V @VVV @VVV\\
0@>>>  \OO_{\PP^1}(-d)@>>> V^\ast\otimes  \OO_{\PP^1}@>>>\T_f(-d)@>>> 0,\end{CD}\end{equation}
i.e. we have obtained a sheaf embedding $\OO_{\PP^1}(d+e+2)\to\T_f$. Since $\deg \T_f=(s+1)d=(s-1)(d+1)+d+e+2$ and the degree of any summand $\OO_{\PP^1}(\delta)$ in a splitting of $\T_f$ is at least $d+1$, we can conclude that $\T_f$ has the form stated in (2).  
\vskip2mm
\paragraph{\em Proof of (i)}
After fixing homogeneous coordinates on $\PP^s$,  the last row of the diagram (\ref{eq:diagram3-->2}) is uniquely determined by the parametrization map $f:\PP^1\to\PP^s$, since this map defines uniquely the sheaf embedding $\OO_{\PP^1}(-d)\to V^\ast\otimes\OO_{\PP^1}$. Hence it is determined by $C$ up to the action of $PGL(2)=\operatorname{Aut}(\PP^1)$. Moreover there exists only one sheaf embedding $\OO_{\PP^1}(e+2)\to\T_f(-d)$ for the given splitting type $\T_f= \OO_{\PP^1}(d+e+2)\oplus\OO_{\PP^1}^{s-1}(d+1)$.  Hence also the sheaf embedding $\mathcal{E}^\ast\to V^\ast\otimes\OO_{\PP^1}$ in the diagram (\ref{eq:diagram3-->2}) is uniquely determined by $C$ up to the action of $PGL(2)$ on $\PP^1$. This means that
 the parametrization map $\PP(\mathcal{E})\to S\subset \PP^s$ is uniquely determined by $C$, up to the (equivariant) action of $PGL(2)$ on $\PP(\mathcal{E})\to \PP^1$. Hence $S$ is uniquely determined by $C$.  
 \vskip2mm
\paragraph{\em Proof of (ii)}
The stated formula for the splitting type of $\N_f$ is an immediate consequence of Proposition \ref{prop:specialnb}.
 \end{proof} 
\begin{rmk}There is a classical connection between the property of a non-degenerate irreducible curve $C$ of sufficiently high degree of being contained in a rational normal scroll and the number of independent quadric hypersurfaces containing $C$. Indeed one has the following result, essentially due  to Castelnuovo.\end{rmk}

 \begin{prop}  A non-degenerate and irreducible curve $C\subset\PP^s$ of degree $d\geq 2s+1$ has $h^0\II_C(2)\leq (s-1)(s-2)/2$. If in addition $C$ is smooth and rational, the equality holds if and only if $C$ is contained in a smooth rational normal scroll of dimension $2$. \end{prop}
 \begin{proof}[Sketch of proof]  Let $\Gamma=C\cap H$ be a general hyperplane section of $C$, which is in general linear position. Then, from the exact sequence $$0\to \II_C(1)\to\II_C(2)\to \II_{\Gamma,H}(2)\to 0,$$ one finds $h^0\II_C(2)\leq h^0\II_{\Gamma,H}(2)$. By a classical argument of Castelnuovo, any $2s-1$ points of $\Gamma$ impose independent conditions on the quadrics of $H\cong\PP^{s-1}$, hence $h^0\II_{\Gamma,H}(2)\leq h^0\OO_H(2)-2s+1=s(s+1)/2-2s+1=(s-1)(s-2)/2$, proving the stated inequality. 
 
 If the equality holds, then $\Gamma$ imposes exactly $2s-1$ conditions on the quadrics of $H\cong\PP^{s-1}$, and since $\deg H\geq 2s+1=2(s-1)+3$ one can apply Castelnuovo's lemma as in p. 531 of \cite{Gr-Ha}, and conclude that $\Gamma$ is contained in a unique rational normal curve of $\PP^{s-1}$. Hence, by the arguments in pp. 531-532 of \cite{Gr-Ha}, the curve $C$ is contained in a rational normal scroll or $s=5$ and $C$ is contained in a Veronese surface in $\PP^5$. When $C$ is a smooth rational curve, we can exclude that $S$ is the Veronese surface $\nu_2(\PP^2)\subset \PP^5$, because any non degenerate smooth curve  $C\subset S$ would come from a smooth curve of degree at least $3$ of $\PP^2$, hence it cannot be rational. Therefore we are left with the case of $S$ a rational normal scroll. As in the proof of the implication (2)$\Rightarrow$(3) of Theorem  \ref{prop:quadricsclassif}, it is easy to see that $S$ is smooth. 
The converse follows from the fact that a rational normal scroll $S\subset\PP^s$ is contained in $(s-1)(s-2)/2$ independent quadrics. \end{proof}
 We conclude this section with a discussion of the relevance of the smoothness assumption in Theorem \ref{prop:quadricsclassif}. Indeed one can see that the implication (3)$\Rightarrow$(2) of Theorem \ref{prop:quadricsclassif} is false if one does not assume $C$ to be smooth. To this purpose, one can find counterexamples already in $\PP^3$. This fact was not explicitly observed in \cite{Eis-VdV1}, where the case $s=3$ of Theorem \ref{prop:quadricsclassif} was proved. Here it is such an example.
 \begin{example}  Let us consider   $g:\PP^1\to\PP^1\times\PP^1$ defined by $$g(u,v)=(u^2:v^2;u^3:v^3)$$ and compose it with the Segre embedding $\PP^1\times\PP^1\to\PP^3$ so to obtain $f:\PP^1\to\PP^3$ defined by $$f(u,v)=(u^5:u^2v^3:v^2u^3:v^5).$$ This is a parametrization of a rational curve $C$ (with two cusps) of degree $5$ contained in the quadric $Q\subset\PP^3$ of equation $x_0x_3-x_1x_2=0$, which is a very simple rational normal scroll, therefore $C$ satisfies (3) of Theorem \ref{prop:quadricsclassif}. Note that $C$ is a curve of divisor class $(2,3)$ in $\PP^1\times\PP^1$, so $C$ is not a section of any of the two $\PP^1$ bundle structures $Q\to\PP^1$. We have, by construction, $$T^\perp=\langle  u^5,u^2v^3,v^2u^3,v^5\rangle.$$ One sees immediately that $T=\langle x^4y,xy^4\rangle$ and therefore $\partial T=\langle x^4,x^3y,xy^3,y^4\rangle$, so that $\dim\partial T=\dim T+2$. Hence, from Proposition \ref{thm:main1} and Definition \ref{def:type} one sees that $T$ has numerical type $(0,0)$ and by Proposition \ref{prop:tangristr} one finds \begin{equation}\label{eq:thm2cex}\T_f=\OO_{\PP^1}^2(7)\oplus\OO_{\PP^1}(6).\end{equation} This contradicts (2) of Theorem \ref{prop:quadricsclassif}.
 Observe that the curve $C$ has not ordinary singularities, but it can be deformed to a rational curve $C'\subset Q$ of divisor class $(2,3)$ with two nodes. Since the vertex $T$ relative to $C$ has numerical type $(0,0)$ and this is the {\em general} numerical type for subspaces $T\subset S^5U$ of $\dim T=2$, then the vertex $T'$ relative to $C'$ will have type $(0,0)$ as well. Hence the restricted tangent sheaf to $C'$ has splitting type as in formula  (\ref{eq:thm2cex}), providing a conterexample to (2) of Theorem \ref{prop:quadricsclassif} by means of a curve with ordinary singularities.  \end{example} 
 \begin{acknowledgment}
 We thank G. Ottaviani and F. Russo for many stimulating and helpful discussions during the development of this work. \end{acknowledgment}

\end{document}